\documentclass[a4paper, intlimits, reqno]{amsart}

\usepackage[english]{babel}
\usepackage[T1]{fontenc}
\usepackage[utf8]{inputenc}

\usepackage{amsmath}
\usepackage{amssymb}
\usepackage{MnSymbol}
\usepackage{amsthm}
\usepackage{amsfonts}
\usepackage{mathrsfs} 
\usepackage{enumitem}
\usepackage{url}
\usepackage{dsfont}
\usepackage[numbers,sort&compress]{natbib}
\usepackage{doi}
\usepackage{prettyref}

\newrefformat{defn}{Definition \ref{#1}}
\newrefformat{rem}{Remark \ref{#1}}
\newrefformat{sect}{Section \ref{#1}}
\newrefformat{prop}{Proposition \ref{#1}}
\newrefformat{thm}{Theorem \ref{#1}}
\newrefformat{cor}{Corollary \ref{#1}}
\newrefformat{ex}{Example \ref{#1}}
\newrefformat{ass}{Assumption \ref{#1}}

\swapnumbers
\newtheoremstyle{dotless}{}{}{\itshape}{}{\bfseries}{}{}{}
\theoremstyle{dotless}
\theoremstyle{plain}
\newtheorem{thm}{Theorem}[section]

\newtheorem{prop}[thm]{Proposition}
\newtheorem{cor}[thm]{Corollary}
\theoremstyle{definition}
\newtheorem{defn}[thm]{Definition}
\newtheorem{rem}[thm]{Remark}
\newtheorem{exa}[thm]{Example}
\newtheorem{ass}[thm]{Assumption}
\newcommand{\N} {\mathbb{N}}
\newcommand{\R} {\mathbb{R}}
\newcommand{\C} {\mathbb{C}}

\DeclareMathOperator{\id}{id}
\DeclareMathOperator{\re}{Re}
\providecommand{\differential}{\mathrm{d}}
\renewcommand{\d}{\differential}
\newcommand\rlim{
\mathchoice{\vcenter{\hbox{${\scriptstyle{+}}$}}}
{\vcenter{\hbox{$\scriptstyle{+}$}}}
{\vcenter{\hbox{$\scriptscriptstyle{+}$}}}
{\vcenter{\hbox{$\scriptscriptstyle{+}$}}}}
\newcommand{\vertiii}[1]{{\left\vert\kern-0.25ex\left\vert\kern-0.25ex\left\vert #1 
    \right\vert\kern-0.25ex\right\vert\kern-0.25ex\right\vert}}

\begin{document}

\title[On equicontinuity and tightness of bi-continuous semigroups]{On equicontinuity and tightness of bi-continuous semigroups}
\author[K.~Kruse]{Karsten Kruse}
\thanks{K.~Kruse acknowledges the support by the Deutsche Forschungsgemeinschaft (DFG) within the Research Training
 Group GRK 2583 ``Modeling, Simulation and Optimization of Fluid Dynamic Applications''.}
\address{(KK) Hamburg University of Technology, Institute of Mathematics,\newline
Am Schwarzenberg-Campus~3,
21073 Hamburg,
Germany}
\email{karsten.kruse@tuhh.de}
\author[F.L.~Schwenninger]{Felix L. Schwenninger}
\address{(FLS) University of Twente, Department of Applied Mathematics, P.O. Box 217, 7500 AE Enschede, The Netherlands, and  University of Hamburg, Center for Optimization and Approximation, Bundesstr.~55, 20146 Hamburg, Germany}
\email{f.l.schwenninger@utwente.nl}



\date{\today}
\begin{abstract}
In contrast to classical strongly continuous semigroups, the study of bi-continuous semigroups comes with some freedom in the properties of the associated locally convex topology. This paper aims to give minimal assumptions in order to recover typical features like tightness and equicontinuity with respect to the mixed topology as well as to carefully clarify on mutual relations between previously studied variants of these notions. The abstract results ---exploiting techniques from topological vector spaces---  are thoroughly discussed by means of several example classes, such as semigroups on spaces of bounded continuous functions. 
\end{abstract}
\maketitle

\section{Introduction}
Strongly continuous semigroups of operators are a well-established framework in the study of evolution equations.
In many applications, however, the semigroups are not strongly continuous ($C_{0}$)
with respect to the norm $\|\cdot\|$ of the underlying Banach space but strongly continuous with respect to a weaker 
Hausdorff locally convex topology $\tau$. Typical examples emerge from the fact that norm-strong continuity is in general not preserved when dual semigroups are considered \cite{vanneerven1992}. Further examples are implemented semigroups 
\cite[Sect.~3.2]{alber2001,bratelli1987}, transition semigroups on the space $\mathrm{C}_{\operatorname{b}}(\Omega)$ of bounded continuous functions on 
e.g.~a Polish space $\Omega$ like the Ornstein--Uhlenbeck semigroup 
\cite{daprato2014,farkas2004,kunze2009} and even on the space of bounded H\"older continuous functions on a 
separable Hilbert space \cite{es_sarhir2006}. In addition, also certain
Koopman semigroups on $\mathrm{C}_{\operatorname{b}}(\Omega)$, i.e.~semigroups induced by a
semiflow on a completely regular Hausdorff space $\Omega$ \cite{dorroh1993,farkas2020}, require such a relaxed notion. 

All these examples belong to the general framework of $\tau$-bi-continuous semigroups, which are $\tau$-strongly continuous and locally bi-equicontinuous, and were first studied by K\"uhnemund in 
\cite{kuehnemund2001,kuehnemund2003}.
On the other hand, 
the theory of $C_{0}$-semigroups was extended beyond Banach spaces to  
equicontinuous $C_{0}$-semigroups on Hausdorff locally convex spaces 
\cite[Chap.~IX]{albanese2016,komatsu1964,yosida1968}, 
quasi-equicontinuous $C_{0}$-semigroups 
\cite{albanese2016,babalola1974,choe1985,miyadera1959,moore1971a,moore1971b}, 
locally equicontinuous $C_{0}$-semigroups \cite{dembart1974,komura1968,ouchi1973} 
and sequentially (locally) equicontinuous semigroups \cite{federico2020}.
The corresponding notion of local bi-equicontinuity of a bi-continuous semigroup is weaker than local 
$\tau$-equicontinuity due to \cite[Examples 6 (a), p.~209-210]{kuehnemund2003}. 
Nevertheless, it was already observed in \cite[Theorem 7.4, p.~180]{kraaij2016} 
that $\tau$-bi-continuous semigroups are locally $\gamma$-equicontinuous 
under some mild assumptions on the mixed topology $\gamma:=\gamma(\|\cdot\|,\tau)$. 
The mixed topology $\gamma$ was introduced in \cite{wiweger1961} and is the finest 
Hausdorff locally convex topology ---even finest linear topology--- 
that coincides with $\tau$ on $\|\cdot\|$-bounded sets. 

In the context of $\tau$-bi-continuous semigroups another notion related to local $\gamma$-equicontinuity 
emerged, namely, the notion of locally equitight $\tau$-bi-continuous semigroups 
(or more generally of families of operators), 
which are sometimes simply referred to as ``tight'' or ``local'' in \cite{es_sarhir2006, farkas2003}.
Equicontinuity and local equicontinuity are crucial ingredients for perturbation results for 
$C_{0}$-semigroups on Hausdorff locally convex spaces like dissipative perturbations 
or Desch--Schappacher perturbations \cite{albanese2016,jacob2015}.
Local equitightness plays a similar role in perturbation theory for bi-continuous semigroups, 
see \cite[Theorem 1.2, p.~669]{es_sarhir2006} (cf.~\cite[Theorem 3.2.3, p.~47]{farkas2003}), 
\cite[Theorems 2.4, 3.2, p.~92, 94-95]{farkas2004}, \cite[Remark 4.1, p.~101]{farkas2004}, 
\cite[Theorem 5, p.~8]{budde2021} (cf.~\cite[Theorem 5.19, p.~81]{budde2019a})
and \cite[Theorem 3.3, p.~582]{budde2021a}. Equitightness is also relevant in ergodic theory for 
$\tau$-bi-continuous semigroups, see \cite[Remark 3.5 (ii), p.~147, Proposition 3.8, p.~150]{albanese2011}. 

In this paper we study $\gamma$-equicontinuity and equitightness for $\tau$-bi-continuous semigroups 
and their relation. The space $(X,\gamma)$ is usually neither barrelled nor bornological, 
since this would imply that $\tau$ coincides 
with the $\|\cdot\|$-topology by \cite[I.1.15 Proposition, p.~12]{cooper1978}. 
Thus automatic local equicontinuity results for strongly continuous semigroups 
like \cite[Proposition 1.1, p.~259]{komura1968} are not applicable. In 
\cite[Corollary 3.10, p.~553]{kunze2009}, \cite[Theorem 4.10, p.~555-556]{kunze2009}
and \cite[Proposition 3.4, p.~161]{kraaij2016} automatic local equicontinuity 
results for strongly continuous semigroups were extended to strong Mackey spaces. 
In particular, barrelled or sequentially complete bornological Hausdorff locally convex spaces 
belong to the class of strong Mackey spaces by \cite[Proposition 3.3 (a), (b), p.~160-161]{kraaij2016}, 
which is not helpful for $(X,\gamma)$.
We take a different route as $(X,\gamma)$ being a strong Mackey space is quite a strong assumption. 
This route is more in the spirit of 
\cite[Proposition 3.3 (c), p.~160-161]{kraaij2016} and \cite[Theorem 7.4, p.~180]{kraaij2016}, 
exploiting the assumption that $(X,\gamma)$ is a C-sequential space, which means that every convex 
sequentially open subset of $(X,\gamma)$ is already open. 

After fixing some notions and preliminaries on $\tau$-bi-continuous semigroups and the mixed topology $\gamma$ 
in \prettyref{sect:notions}, 
we derive sufficient conditions on the interplay of $\|\cdot\|$, $\tau$ and $\gamma$ that imply 
the quasi-$\gamma$-equicontinuity of a $\tau$-bi-continuous semigroup in \prettyref{sect:continuity}. 
Moreover, we deduce sufficient conditions that guarantee 
the equivalence between (local, quasi-)$\gamma$-equicontinuity and (local, quasi-)equitightness,   \prettyref{thm:mixed_equicont_main} in combination with \prettyref{prop:mixed_equicont}.
It turns out that these conditions are satisfied
for most of the classical examples of $\tau$-bi-continuous semigroups we mentioned before. Indeed, these applications are covered by the results in 
\prettyref{sect:applications}.

\section{Notions and preliminaries}
\label{sect:notions}

For a continuous map $f\colon (X_{1},\tau_{1})\to (X_{2},\tau_{2})$ from a topological space $(X_{1},\tau_{1})$ to 
a topological space $(X_{2},\tau_{2})$ we typically write that it is $\tau_{1}$-$\tau_{2}$-continuous and, 
if $(X_{2},\tau_{2})=(X_{1},\tau_{1})$, we just write that it is $\tau_{1}$-continuous. 
For two topologies $\tau_{1}$ and $\tau_{2}$ on a space $X$, we write $\tau_{1}\leq\tau_{2}$ if the topology 
$\tau_{1}$ is coarser than $\tau_{2}$. 
For a vector space $X$ over the field $\R$ or $\C$ with a Hausdorff locally convex topology $\tau$ 
we denote by $(X,\tau)'$ the topological linear dual space and just write $X':=(X,\tau)'$ 
if $(X,\tau)$ is a Banach space.
We use the symbol $\mathcal{L}(X;Y):=\mathcal{L}((X,\|\cdot\|_{X});(Y,\|\cdot\|_{Y}))$ 
for the space of continuous linear operators from a Banach space $(X,\|\cdot\|_{X})$ 
to a Banach space $(Y,\|\cdot\|_{Y})$ and denote by $\|\cdot\|_{\mathcal{L}(X;Y)}$ the operator norm on 
$\mathcal{L}(X;Y)$. If $X=Y$, we set $\mathcal{L}(X):=\mathcal{L}(X;X)$.

\begin{ass}[{\cite[Assumptions 1, p.~206]{kuehnemund2003}}]\label{ass:standard}
Let $(X,\|\cdot\|,\tau)$ be a triple where $(X,\|\cdot\|)$ is a Banach space, and 
\begin{enumerate}
\item[(i)] $\tau$ is a coarser Hausdorff locally convex topology than the $\|\cdot\|$-topology,
\item[(ii)] $\tau$ is sequentially complete on $\|\cdot\|$-bounded sets, 
i.e.~every $\|\cdot\|$-bounded $\tau$-Cauchy sequence is $\tau$-convergent, 
\item[(iii)] the dual space $(X,\tau)'$ is \emph{norming}, i.e. 
\[
 \|x\|=\sup_{\substack{y\in(X,\tau)'\\ \|y\|_{X'}\leq 1}}|y(x)|,\quad x\in X.
\]
\end{enumerate}
\end{ass} 

For what follows, the mixed topology, \cite[Section 2.1]{wiweger1961}, and the notion of a Saks space 
\cite[I.3.2 Definition, p.~27-28]{cooper1978} will be crucial.

\begin{defn}\label{defn:mixed_top_Saks}
Let $(X,\|\cdot\|)$ be a Banach space and $\tau$ a Hausdorff locally convex topology on $X$ that is coarser 
than the $\|\cdot\|$-topology $\tau_{\|\cdot\|}$.  Then
\begin{enumerate}
	\item[(a)]  the \emph{mixed topology} $\gamma := \gamma(\|\cdot\|,\tau)$ is
	the finest linear topology on $X$ that coincides with $\tau$ on $\|\cdot\|$-bounded sets and such that 
	$\tau\leq \gamma \leq \tau_{\|\cdot\|}$; 
    \item[(b)] the triple $(X,\|\cdot\|,\tau)$ is called a \emph{Saks space} if there exists a directed system 
    of seminorms $\mathcal{P}_{\tau}$ that generates the topology $\tau$ such that
\begin{equation}\label{eq:saks}
\|x\|=\sup_{p\in\mathcal{P}_{\tau}} p(x), \quad x\in X.
\end{equation}
\end{enumerate}
\end{defn}

In the next remark we collect some observations from \cite[Remark 5.3, p.~2680]{kruse_meichnser_seifert2018},
\cite[Lemma 5.5 (b), p.~2681]{kruse_meichnser_seifert2018} and \cite[Remark 5.2, p.~338]{budde2019} 
(see \cite[Lemma 4.4, p.~163]{kraaij2016} as well) concerning the previous assumptions.

\begin{rem}\label{rem:mixed_top}
\begin{enumerate}
\item[(a)] The mixed topology is Hausdorff locally convex and our definition is equivalent to the one from the 
literature \cite[Section 2.1]{wiweger1961} due to \cite[Lemmas 2.2.1, 2.2.2, p.~51]{wiweger1961}.
\item[(b)] Let $\tau$ be a Hausdorff locally convex topology on $X$ that is coarser than the $\|\cdot\|$-topology 
and $(X,\tau)'$ norming. Then the sequential completeness of $\tau$ on $\|\cdot\|$-bounded sets is 
equivalent to the sequential completeness of $(X,\gamma)$.
\item[(c)] The existence of a system of seminorms $\mathcal{P}_{\tau}$ generating the topology $\tau$ 
and satisfying \eqref{eq:saks} is equivalent to the property that $(X,\tau)'$ is norming. 
We note however  that not every system of seminorms that induces the topology $\tau$ need to
fulfil \prettyref{eq:saks} even if $(X,\tau)'$ is norming (see \cite[Example B), p.~65]{wiweger1961}).
\item[(d)] \prettyref{ass:standard} (iii) may be weakened to 
\[
 \|x\|=\sup_{y\in \Phi(\tau)}|y(x)|,\quad x\in X,
\]
where $\Phi(\tau)$ is the set all linear functionals $y\in X'$ with $\|y\|_{X'}\leq 1$ 
whose restriction to the unit ball $B_{\|\cdot\|}:=\{x\in X\;|\; \|x\|\leq 1\}$ is $\tau$-sequentially continuous 
by \cite[Remarks 2.2, 2.4, p.~4]{kunstmann2020}. However, no concrete example is known where this is strictly weaker 
than \prettyref{ass:standard} (iii).
\item[(e)] With the previously stated facts in this remark, we can rephrase \prettyref{ass:standard} 
by saying that $(X,\|\cdot\|,\tau)$ is a Saks space such that $(X,\gamma)$ is sequentially complete.
\end{enumerate}
\end{rem}

\begin{exa}\label{ex:mixed_top}
\begin{enumerate}
\item[(a)] Let $\Omega$ be a \emph{completely regular} Hausdorff space. 
We recall that a topological space $\Omega$ is called completely regular 
if for any non-empty closed subset $A\subset\Omega$ and 
$x\in\Omega\setminus A$ there is a continuous function $f\colon\Omega\to[0,1]$ 
such that $f(x)=0$ and $f(z)=1$ for all $z\in A$ (see \cite[Definition 11.1, p.~180]{james1999}). 
Let $\mathrm{C}_{\operatorname{b}}(\Omega)$ be the space of bounded real- or complex-valued continuous functions 
on $\Omega$ and  
\[
\|f\|_{\infty}:=\sup_{x\in\Omega}|f(x)|,\quad f\in \mathrm{C}_{\operatorname{b}}(\Omega).
\]
The compact-open topology, i.e.~the topology $\tau_{\operatorname{co}}$ of uniform convergence on compact subsets 
of $\Omega$, is induced by the directed system of seminorms $\mathcal{P}_{\tau_{\operatorname{co}}}$ given by 
\[
p_{K}(f):=\sup_{x\in K}|f(x)|,\quad f\in \mathrm{C}_{\operatorname{b}}(\Omega),
\]
for compact $K\subset \Omega$. 
Then $(\mathrm{C}_{\operatorname{b}}(\Omega),\|\cdot\|_{\infty},\tau_{\operatorname{co}})$ is a Saks space 
by \cite[Example D), p.~65-66]{wiweger1961}. 

Let $\mathcal{V}$ denote the set of all non-negative bounded functions $\nu$ on $\Omega$ 
that vanish at infinity, i.e.~for every $\varepsilon>0$ the set $\{x\in\Omega\;|\;\nu(x)\geq\varepsilon\}$ is compact. 
Let $\beta_{0}$ be the Hausdorff locally convex topology on $\mathrm{C}_{\operatorname{b}}(\Omega)$ that is induced 
by the seminorms 
\[
|f|_{\nu}:=\sup_{x\in\Omega}|f(x)|\nu(x),\quad f\in\mathrm{C}_{\operatorname{b}}(\Omega),
\]
for $\nu\in\mathcal{V}$. Due to \cite[Theorem 2.4, p.~316]{sentilles1972} 
we have $\gamma(\|\cdot\|_{\infty},\tau_{\operatorname{co}})=\beta_{0}$. 
If $\Omega$ is locally compact, then $\mathcal{V}$ may be replaced by the functions 
in $\mathrm{C}_{0}(\Omega)$ that are non-negative by \cite[Theorem 2.3 (b), p.~316]{sentilles1972} where 
$\mathrm{C}_{0}(\Omega)$ is the space of real-valued continuous functions on $\Omega$ that vanish at infinity.
\item[(b)] Let $(X,\|\cdot\|)$ be a Banach space and recall
\[
\|y\|_{X'}=\sup_{x\in B_{\|\cdot\|}}|y(x)|,\quad y\in X'.
\]
The weak$^{\ast}$-topology $\sigma^{\ast}:=\sigma(X',X)$ is induced by the directed system of seminorms given by 
\[
p_{K}(y):=\sup_{x\in K}|y(x)|,\quad y\in X',
\]
for finite $K\subset  B_{\|\cdot\|}$. 

Then $(X',\|\cdot\|_{X'},\sigma^{\ast})$ is a Saks space and 
$\gamma(\|\cdot\|_{X'},\sigma^{\ast})=\tau_{\operatorname{c}}(X',X)$
by \cite[Example E), p.~66]{wiweger1961} where $\tau_{\operatorname{c}}(X',X)$ is the topology of uniform convergence 
on compact subsets of $X$.
\item[(c)] Let $(X,\|\cdot\|)$ be a Banach space.
The dual Mackey-topology $\mu^{\ast}:=\mu(X',X)$ is induced by the directed system of seminorms given by 
\[
p_{K}(y):=\sup_{x\in K}|y(x)|,\quad y\in X',
\]
for $\sigma(X,X')$-compact absolutely convex $K\subset X$. Since $(X',\mu^{\ast})'=X$ is norming, 
$(X',\|\cdot\|_{X'},\mu^{\ast})$ is a Saks space by \prettyref{rem:mixed_top} (c). 
Furthermore, we have $\gamma(\|\cdot\|_{X'},\mu^{\ast})=\mu^{\ast}$ by the second example 
in \cite[p.~593]{conradie2006} as $\gamma(\|\cdot\|_{X'},\mu^{\ast})$ is, in particular, the finest 
Hausdorff locally convex topology that coincides with $\mu^{\ast}$ on $\|\cdot\|_{X'}$-bounded sets
by \prettyref{rem:mixed_top} (a).
\item[(d)] Let $(X,\|\cdot\|_{X})$ and $(Y,\|\cdot\|_{Y})$ be Banach spaces and recall
\[
\|R\|_{\mathcal{L}(X;Y)}=\sup_{x\in B_{\|\cdot\|_{X}}}\|Rx\|_{Y},\quad R\in \mathcal{L}(X;Y).
\]
The weak operator topology $\tau_{\operatorname{wot}}$ on $\mathcal{L}(X;Y)$ is 
induced by the directed system of seminorms given by 
\[
p_{N,M}(R):=\sup_{x\in N, y'\in M}|y'(Rx)|,\quad R\in \mathcal{L}(X;Y),
\]
for finite $N\subset X$ and finite $M\subset Y'$. 
The strong operator topology $\tau_{\operatorname{sot}}$ on $\mathcal{L}(X;Y)$ is induced by the directed system 
of seminorms given by 
\[
p_{N}(R):=\sup_{x\in N}\|Rx\|_{Y},\quad R\in \mathcal{L}(X;Y),
\]
for finite $N\subset X$. Due to \cite[p.~75]{kuehnemund2001} and \cite[VI.1.4 Theorem, p.~477]{dunford1958} 
$(\mathcal{L}(X;Y),\tau_{\operatorname{wot}})'=(\mathcal{L}(X;Y),\tau_{\operatorname{sot}})'$ is 
norming for $(\mathcal{L}(X;Y),\|\cdot\|_{\mathcal{L}(X;Y)})$ and thus 
$(\mathcal{L}(X;Y),\|\cdot\|_{\mathcal{L}(X;Y)},\tau_{\operatorname{wot}})$ and 
$(\mathcal{L}(X;Y),\|\cdot\|_{\mathcal{L}(X;Y)},\tau_{\operatorname{sot}})$ are Saks spaces 
by \prettyref{rem:mixed_top} (c). 
\end{enumerate}
\end{exa}

Concerning example (a), we note that examples of completely regular Hausdorff spaces 
are locally compact Hausdorff spaces by \cite[3.3.1 Theorem, p.~148]{engelking1989}, 
uniformisable, particularly metrisable, spaces 
by \cite[Proposition 11.5, p.~181]{james1999} and Hausdorff locally convex spaces 
by \cite[Proposition 3.27, p.~95]{fabian2011}.

\begin{defn}[{\cite[Definition 3, p.~207]{kuehnemund2003}}]\label{defn:bi_continuous}
Let $(X,\|\cdot\|,\tau)$ be a triple satisfying \prettyref{ass:standard}. 
A family $(T(t))_{t\geq 0}$ in $\mathcal{L}(X)$ is called a $\tau$\emph{-bi-continuous semigroup} if 
\begin{enumerate}
\item[(i)] $(T(t))_{t\geq 0}$ is a \emph{semigroup}, i.e.~$T(t+s)=T(t)T(s)$ and $T(0)=\id$ for all $t,s\geq 0$,
\item[(ii)] $(T(t))_{t\geq 0}$ is $\tau$\emph{-strongly continuous}, 
i.e.~the map $T_{x}\colon[0,\infty)\to(X,\tau)$, $T_{x}(t):=T(t)x$, is continuous for all $x\in X$, 
\item[(iii)] $(T(t))_{t\geq 0}$ is \emph{exponentially bounded} (of type $\omega$), i.e.~there exist $M\geq 1$ and 
$\omega\in\R$ such that $\|T(t)\|_{\mathcal{L}(X)}\leq Me^{\omega t}$ for all $t\geq 0$,
\item[(iv)] $(T(t))_{t\geq 0}$ is \emph{locally bi-equicontinuous}, 
i.e.~for every sequence $(x_n)_{n\in\N}$ in $X$, $x\in X$ 
with $\sup\limits_{n\in\N}\|x_n\|<\infty$ and $\tau\text{-}\lim\limits_{n\to\infty} x_n = x$ it holds that
     \[
      \tau\text{-}\lim_{n\to\infty} T(t)(x_n-x) = 0
     \]
locally uniformly for all $t\in [0,\infty)$.
\end{enumerate}
\end{defn}

For a $\tau$-bi–continuous semigroup $(T(t))_{t\geq 0}$ we call
\[
\omega_{0}:=\omega_{0}(T):=\inf\{\omega\in\R\;|\; \exists\; M\geq 1\;\forall\;t\geq 0:\;\|T(t)\|_{\mathcal{L}(X)}
                                 \leq Me^{\omega t}\}
\]
its \emph{growth bound} (see \cite[p.~7]{kuehnemund2001}).

\begin{rem}\label{rem:mixed_top_bi_cont}
Let $(X,\|\cdot\|)$ be a Banach space, $\tau$ a Hausdorff locally convex topology on $X$ that is coarser 
than the $\|\cdot\|$-topology, and $\gamma:=\gamma(\|\cdot\|,\tau)$ the mixed topology.
\begin{enumerate}
\item[(a)] Due to \cite[I.1.10 Proposition, p.~9]{cooper1978} a sequence $(x_{n})_{n\in\N}$ in $X$ is 
$\gamma$-convergent if and only if it is $\tau$-convergent and $\|\cdot\|$-bounded. This implies that 
$\tau$-bi-continuous semigroups are $\gamma$-strongly continuous and locally sequentially $\gamma$-equicontinuous.
\item[(b)] The condition of exponential boundedness in \prettyref{defn:bi_continuous} is superfluous. 
Indeed, if $(T(t))_{t\geq 0}$ is a semigroup of linear 
operators on $X$ that is $\gamma$-strongly continuous and locally sequentially $\gamma$-equicontinuous, 
then $(T(t))_{t\geq 0}$ is exponentially bounded by \cite[Proposition 3.6 (ii), p.~1137]{federico2020} 
because a set in $X$ is $\gamma$-bounded if and only if it is $\|\cdot\|$-bounded 
by \cite[2.4.1 Corollary, p.~56]{wiweger1961}. 
\item[(c)] Let $(X,\|\cdot\|,\tau)$ satisfy \prettyref{ass:standard}.
It follows from (a) and (b) that a family of operators $(T(t))_{t\geq 0}$ in $\mathcal{L}(X)$ is 
a $\tau$-bi-continuous semigroup if and only if it is a $\gamma$-strongly continuous 
and locally sequentially $\gamma$-equicontinuous semigroup. This remains valid if $\gamma$ is replaced by any other 
Hausdorff locally convex topology on $X$ that has the same convergent sequences as $\gamma$ 
(cf.~\cite[Lemma A.1.2, p.~72]{farkas2003} and \cite[Proposition 1.6, p.~313]{farkas2011}).
\end{enumerate}
\end{rem}

\section{Continuity, equicontinuity and tightness}
\label{sect:continuity}

We briefly recall the different types of equicontinuity that emerged in the context of semigroups.

\begin{defn}
Let $(X_{1},\tau_{1})$ and $(X_{2},\tau_{2})$ be Hausdorff locally convex spaces. 
A family $(T(t))_{t\in I}$ of maps from a set $M_{1}\subset X_{1}$ to $X_{2}$ 
is called ${\tau_{1}}_{\mid M_{1}}$-$\tau_{2}$\emph{-equicontinuous at} $x\in M_{1}$ if 
for every $\tau_{2}$-neighbourhood $U_{2}$ of zero in $X_{2}$ there is a $\tau_{1}$-neighbourhood $U_{1}$ of $x$ in 
$X_{1}$ such that $T(t)(U_{1}\cap M_{1})\subset T(t)(x)+U_{2}$ for all $t\in I$. 
The family $(T(t))_{t\in I}$ is called ${\tau_{1}}_{\mid M_{1}}$-$\tau_{2}$\emph{-equicontinuous on} $M_{1}$ 
if it is ${\tau_{1}}_{\mid M_{1}}$-$\tau_{2}$-equicontinuous at every $x\in M_{1}$. 
If $(X_{2},\tau_{2})=(X_{1},\tau_{1})$ and $M_{1}=X_{1}$, then we just write $\tau_{1}$-equicontinuous 
instead of ${\tau_{1}}_{\mid M_{1}}$-$\tau_{2}$-equicontinuous on $M_{1}$. 
\end{defn}

\begin{rem}
Let $\mathcal{P}_{\tau_{1}}$ and $\mathcal{P}_{\tau_{2}}$ be directed systems of seminorms inducing the 
topologies $\tau_{1}$ and $\tau_{2}$, respectively. 
If $(T(t))_{t\in I}$ is a family of linear maps $T(t)\colon X_{1}\to X_{2}$, then it is 
$\tau_{1}$-$\tau_{2}$-equicontinuous if and only if 
\[
\forall\;p\in\mathcal{P}_{\tau_{2}}\;\exists\;\widetilde{p}\in\mathcal{P}_{\tau_{1}},\;C\geq 0\;
\forall\;t\in I,\,x\in X_{1}:\;p(T(t)x)\leq C\widetilde{p}(x).
\]
\end{rem}

\begin{defn}
Let $(X,\tau)$ be a Hausdorff locally convex space and $(T(t))_{t\geq 0}$ a family of 
linear maps $X\to X$.
\begin{enumerate}
\item[(a)] $(T(t))_{t\geq 0}$ is called \emph{locally} $\tau$\emph{-equicontinuous} 
if $(T(t))_{t\in[0,t_{0}]}$ is $\tau$-equicontinu\-ous for all $t_{0}\geq 0$.
\item[(b)] $(T(t))_{t\geq 0}$ is called \emph{quasi-}$\tau$\emph{-equicontinuous} 
if there exists $\alpha\in\R$ such that $(e^{-\alpha t}T(t))_{t\geq 0}$ is $\tau$-equicontinuous.
\end{enumerate}
\end{defn}

$\tau$-equicontinuity of $(T(t))_{t\geq 0}$ implies quasi-$\tau$-equicontinuity, which implies local 
$\tau$-equicontinuity, which again implies the $\tau$-continuity of all $T(t)$, $t\geq 0$. 
Due to \cite[Example 3.2, p.~549]{kunze2009} the left translation semigroup on $\mathrm{C}_{\operatorname{b}}(\R)$ 
is locally $\tau_{\operatorname{co}}$-equicontinuous but not quasi-$\tau_{\operatorname{co}}$-equicontinuous.
The Gau{\ss}--Weierstra{\ss} semigroup on $\mathrm{C}_{\operatorname{b}}(\R^{d})$ 
is $\tau_{\operatorname{co}}$-bi-continuous but not locally $\tau_{\operatorname{co}}$-equicontinuous 
by \cite[Examples 6 (a), p.~209-210]{kuehnemund2003}.
However, we will see in \prettyref{ex:strict_topo_semigroups} that both semigroups are 
locally $\beta_{0}$-equicontinuous, even quasi-$\beta_{0}$-equicontinuous, where we recall that $\beta_{0}=\gamma(\|\cdot\|_{\infty},\tau_{\operatorname{co}})$.

\begin{defn}\label{defn:equitight}
Let $(X,\|\cdot\|,\tau)$ be a Saks space and $\mathcal{P}_{\tau}$ a directed system of 
seminorms generating the topology $\tau$. A family of linear maps $(T(t))_{t\in I}$ from $X$ to $ X$ is called 
$(\|\cdot\|,\tau)$\emph{-equitight}
if 
\[
\forall\;\varepsilon>0,\,p\in\mathcal{P}_{\tau}\;\exists\;\widetilde{p}\in\mathcal{P}_{\tau},\,C\geq 0\;
\forall\;t\in I,\,x\in X:\;p(T(t)x)\leq C \widetilde{p}(x)+\varepsilon\|x\|.
\]
If $I$ is a singleton, i.e.~$I=\{t\}$ for some $t$, we just call $T(t)$ $(\|\cdot\|,\tau)$\emph{-tight}. 
If no confusion seems to be likely, we just write equitight and tight instead of $(\|\cdot\|,\tau)$-equitight 
and $(\|\cdot\|,\tau)$-tight, respectively.
\end{defn}

We note that the definition of (equi-)tightness does not depend on the choice of the directed 
system of seminorms $\mathcal{P}_{\tau}$ that generates the topology $\tau$. 

Tight operators $T\in\mathcal{L}(X)$ as well as families of equitight operators 
$(T(t))_{t\in [0,t_{0}]}$ in $\mathcal{L}(X)$ were introduced in \cite[Definitions 1.2.20, 1.2.21, p.~12]{farkas2003} 
under the name \emph{local} (see \cite[Definition 5.13, p.~79]{budde2019a} and \cite[Definition 5, p.~5]{budde2021} 
as well).
In the context of $\tau$-bi-continuous semigroups $(T(t))_{t\geq 0}$ the notion of tightness is used in 
\cite[Definition 1.1, p.~668]{es_sarhir2006}, meaning that $(T(t))_{t\in [0,t_{0}]}$ is equitight (or 
local) for all $t_{0}\geq 0$. The name tightness stems from the relation to tight measures in the context of 
$\tau_{\operatorname{co}}$-bi-continuous semigroups on $\mathrm{C}_{\operatorname{b}}(\Omega)$, $\Omega$ \emph{Polish} 
(i.e.~completely metrisable and separable), see \cite[Lemma 2.3, Theorem 2.4, p.~92]{farkas2004}, 
\cite[Proposition 3.3, p.~317]{farkas2011} and \prettyref{rem:mackey_mazur} (a). 
We prefer the notion of tightness to localness in varying degrees in correspondence to (equi-)continuity 
and thus introduce the following notions of local and quasi-equitightness. 

\begin{defn}
Let $(X,\|\cdot\|,\tau)$ be a Saks space
and $(T(t))_{t\geq 0}$ a family of linear maps $X\to X$. 
\begin{enumerate}
\item[(a)] $(T(t))_{t\geq 0}$ is called \emph{locally equitight} if $(T(t))_{t\in [0,t_{0}]}$ is 
equitight for all $t_{0}\geq 0$.
\item[(b)] $(T(t))_{t\geq 0}$ is called \emph{quasi-equitight} if there is $\alpha\in\R$ such that 
$(e^{-\alpha t}T(t))_{t\geq 0}$ is equitight.
\end{enumerate}
\end{defn}

Equitightness of $(T(t))_{t\geq 0}$ implies quasi-equitightness, which implies local 
equitightness, which again implies the tightness of all $T(t)$, $t\geq 0$. 
We remark that the notion of tightness is not only relevant for bi-continuous semigroups but for other operators 
as well. In \cite[Lemma 1.2.23, p.~12]{farkas2003} it is shown that the resolvent family 
$\{R(\lambda,A)\;|\;\lambda\in [\alpha,\beta]\}$ for $\beta\geq\alpha>\omega$ of a 
locally equitight $\tau$-bi-continuous semigroup $(T(t))_{t\geq 0}$ of type $\omega$ is equitight 
as well.\footnote{The assumption that $(T(t))_{t\geq 0}$ should be locally equitight 
(i.e.~local in terms of \cite{farkas2003}) is missing in \cite[Lemma 1.2.23, p.~12]{farkas2003} 
but is used in its proof.} We note the following slight improvement (and correction) of this statement. 

\begin{prop}\label{prop:resolv_equitight}
Let $(X,\|\cdot\|,\tau)$ be a triple satisfying \prettyref{ass:standard} 
and $(T(t))_{t\geq 0}$ a $\tau$-bi-continuous semigroup.
If $(T(t))_{t\geq 0}$ is locally equitight, then $\{R(\lambda,A)\;|\;\re\lambda\in[\alpha,\beta]\}$ 
is equitight for all $\beta\geq\alpha>\omega_{0}(T)$. 
\end{prop}
\begin{proof}
First, we note that \cite[Lemma 1.2.23, p.~12]{farkas2003}, which is based on representing the 
resolvent as a Laplace transform \cite[Definition 9, p.~213]{kuehnemund2003},
still holds if the type $\omega$ is replaced by the growth bound $\omega_{0}(T)$ 
and the condition $\lambda>\omega$ by $\re\lambda>\omega_{0}(T)$ 
since the estimates in the proof on \cite[p.~13]{farkas2003} still hold if $\lambda$ is replaced 
by $\re\lambda$ after the first ``$\leq$''-sign.
Then it follows from this updated version of \cite[Lemma 1.2.23, p.~12]{farkas2003} and 
\cite[Definition 1.2.21, p.~12]{farkas2003} 
that $\{R(\lambda,A)\;|\;\re\lambda\in[\alpha,\beta]\}$ is equitight 
for all $\beta\geq\alpha>\omega_{0}(T)$.
\end{proof}

The following example shows that one cannot drop the condition of local equi\-tightness of $(T(t))_{t\geq 0}$
in \prettyref{prop:resolv_equitight}.

\begin{exa}\label{ex:counterex}
We use the example of a $\tau_{\operatorname{co}}$-bi-continuous semigroup 
$(T(t))_{t\geq 0}$ on the space $\mathrm{C}_{\operatorname{b}}(\Omega)$ of bounded $\R$-valued continuous 
functions on $\Omega$ from \cite[Example 4.1, p.~320]{farkas2011} 
where $\Omega:=[0,\omega_{1})$ is equipped with the order topology and $\omega_{1}$ is the first uncountable ordinal. Let $\beta\Omega$ be the Stone--\v{C}ech compactification of $\Omega$, i.e.~$\beta\Omega=[0,\omega_{1}]$, and choose 
$x\in\beta\Omega$ as constructed in \cite[Example 4.1, p.~320]{farkas2011}, i.e.~$x=\omega_{1}$.
Then the operator $A:=\mathbf{1}\otimes\delta_{x}$, i.e.~$Af(z)=f(x)$, $z\in\Omega$, for 
$f\in\mathrm{C}_{\operatorname{b}}(\Omega)$ with $f$ continuously extended to $\beta\Omega$, 
belongs to $\mathcal{L}(\mathrm{C}_{\operatorname{b}}(\Omega))$ and generates the 
$\tau_{\operatorname{co}}$-bi-continuous semigroup on $\mathrm{C}_{\operatorname{b}}(\Omega)$ given by 
\[
T(t):=\operatorname{id}-A+e^{t}A,\quad t\geq 0,
\]
by \cite[Example 4.1, p.~320]{farkas2011}. Further,
the operators $T(t)$ for $t>0$ are not $\tau_{\operatorname{co}}$-continuous on $\|\cdot\|_{\infty}$-bounded sets 
by \cite[Example 4.1, p.~320]{farkas2011} and thus the semigroup $(T(t))_{t\geq 0}$ is not 
locally $(\|\cdot\|_{\infty},\tau_{\operatorname{co}})$-equitight by \prettyref{prop:mixed_equicont} below.
We have 
\[
 \|T(t)\|_{\mathcal{L}(\mathrm{C}_{\operatorname{b}}(\Omega))}
=\sup_{\substack{f\in\mathrm{C}_{\operatorname{b}}(\Omega)\\ \|f\|_{\infty}\leq 1}}\|T(t)f\|_{\infty}
=\sup_{\substack{f\in\mathrm{C}_{\operatorname{b}}(\Omega)\\ \|f\|_{\infty}\leq 1}}
 \sup_{z\in\Omega}|f(z)-f(x)+e^{t}f(x)|
\leq 3e^{t},\quad t\geq 0,
\]
thus $\omega_{0}(T)\leq 1$, and 
\begin{align*}
 R(\lambda,A)f
&=\int_{0}^{\infty}e^{-\lambda s}T(s)f\d s
 =\int_{0}^{\infty}e^{-\lambda s}(f-f(x)+e^{s}f(x))\d s\\
&=(f-f(x))\int_{0}^{\infty}e^{-\lambda s}\d s +f(x)\int_{0}^{\infty}e^{(1-\lambda)s}\d s 
 =\frac{f-f(x)}{\lambda}+\frac{f(x)}{\lambda-1}
\end{align*}
for $\re\lambda >1$ and $f\in\mathrm{C}_{\operatorname{b}}(\Omega)$ 
by \cite[Definition 9, p.~213]{kuehnemund2003} where the first two integrals 
are to be understood as improper $\tau_{\operatorname{co}}$-Riemann integrals. 

Now, suppose that $\{R(\lambda,A)\;|\;\lambda\in[2,3]\}$ is 
$(\|\cdot\|_{\infty},\tau_{\operatorname{co}})$-equitight. Let $K\subset\Omega$ be compact and 
choose $\varepsilon:=\frac{1}{n}$ for $n\in\N$. Then there are a compact set 
$\widetilde{K}_{K,n}\subset\Omega$ and $C_{K,n}\geq 0$ such that 
for all $\lambda\in[2,3]$ and $f\in\mathrm{C}_{\operatorname{b}}(\Omega)$ it holds that
\begin{equation}\label{eq:counterex}
\sup_{z\in K}\Bigl|\frac{1}{\lambda}f(z)+\frac{1}{\lambda(\lambda-1)}f(x)\Bigr|
=\sup_{z\in K}|R(\lambda,A)f(z)|
\leq C_{K,n}\sup_{z\in\widetilde{K}_{K,n}}|f(z)|+\frac{1}{n}\|f\|_{\infty}.
\end{equation}
The set $A_{K,n}:=(K\cup\widetilde{K}_{K,n})\subset\Omega=[0,\omega_{1})$ is a compact subset of 
$\beta\Omega=[0,\omega_{1}]$ for every $n\in\N$ and $x=\omega_{1}\notin[0,\omega_{1})$. 
Since $[0,\omega_{1}]$ is a completely regular Hausdorff space 
by \cite[Part II, 43.~Closed ordinal space {$[0,\Omega]$}; 4, p.~69]{steen1978}, for every $n\in\N$ there exists a continuous function $f_{n}\colon [0,\omega_{1}]\to [0,1]$ such that 
$f_{n}=0$ on $A_{K,n}$ and $f_{n}(x)=1$ by \cite[(2.1.5) Proposition, p.~17]{buchwalter1969}. 
This implies that
\[
 \frac{1}{6}
\leq\frac{1}{\lambda(\lambda-1)}
=\sup_{z\in K}\Bigl|\frac{1}{\lambda}f_{n}(z)+\frac{1}{\lambda(\lambda-1)}f_{n}(x)\Bigr|
\underset{\eqref{eq:counterex}}{\leq} C_{K,n}\sup_{z\in\widetilde{K}_{K,n}}|f_{n}(z)|+\frac{1}{n}\|f_{n}\|_{\infty}
=\frac{1}{n}
\]
for all $n\in\N$ and $\lambda\in[2,3]$, which is a contradiction.
\end{exa}

Local equitightness plays a crucial role in perturbation results for bi-continuous semigroups, 
see \cite[Theorem 1.2, p.~669]{es_sarhir2006} (cf.~\cite[Theorem 3.2.3, p.~47]{farkas2003}), 
\cite[Theorems 2.4, 3.2, p.~92, 94-95]{farkas2004}, \cite[Remark 4.1, p.~101]{farkas2004}, 
\cite[Theorem 5, p.~8]{budde2021} (cf.~\cite[Theorem 5.19, p.~81]{budde2019a})
and \cite[Theorem 3.3, p.~582]{budde2021a}.
Unfortunately, the missing assumption of local equitightness in \cite[Lemma 1.2.23, p.~12]{farkas2003} 
has some consequences for one of the cited references above. 

\begin{rem}\label{rem:consequences_of_missing_assump}
We refer to \cite{budde2021} for the relevant notions. In \cite[Theorem 5, p.~8]{budde2021} 
(cf.~\cite[Theorem 5.19, p.~81]{budde2019a}) we have the following assumptions.
Let $(X,\|\cdot\|,\tau)$ be a triple satisfying \prettyref{ass:standard}, 
$(A,D(A))$ the generator of a positive, locally equitight (local in terms of \cite{budde2021}) 
$\tau$-bi-continuous semigroup $(T(t))_{t\geq 0}$ on a bi-AL space $X$ with $\eta$-bi-dense domain $D(A)$ for some
$1<\eta<2$. Let $B\colon D(A)\to X$ be a positive operator, i.e.~$Bx \geq 0$ for each $x\in D(A)\cap X_{+}$, 
and assume that $BR(\lambda, A)$ is tight and $(A+B,D(A))$ is resolvent positive.

In the proof of \cite[Theorem 5, p.~8]{budde2021} it is used in \cite[p.~14]{budde2021} 
that for $s\in[0,1]$ the tight operator $sBR(\lambda, A)\in\mathcal{L}(X)$ generates
a $\tau$-bi-continuous and $\|\cdot\|$-strongly continuous semigroup $(E(t))_{t\geq 0}$ on $X$ given by 
\[
E(t)=e^{tsBR(\lambda,A)}=\sum_{n=0}^{\infty}\frac{t^{n}(sBR(\lambda,A))^{n}}{n!},\quad t\geq 0.
\]
Then \cite[Lemma 1.2.23, p.~12]{farkas2003} is applied to the semigroup $(E(t))_{t\geq 0}$ to conclude 
that $R(1,sBR(\lambda,A))$ is tight. But it seems not to be guaranteed that this conclusion is true 
due to the missing assumption in \cite[Lemma 1.2.23, p.~12]{farkas2003}. However, the conclusion is 
true if $(E(t))_{t\geq 0}$ is supposed to be locally equitight by \prettyref{prop:resolv_equitight}.
For instance, this additional assumption can be ensured by \prettyref{thm:mixed_equicont_main} (b) below which implies that every 
$\tau$-bi-continuous semigroup on $X$ is locally equitight
if the space $(X,\gamma)$ with mixed topology $\gamma:=\gamma(\|\cdot\|,\tau)$ 
is \emph{C-sequential}, meaning that every convex sequentially 
open subset of $(X,\gamma)$ is already open, and if $\gamma$ coincides with the \emph{submixed topology} $\gamma_{s}$ 
(see \prettyref{defn:submixed_top} below). 

Concerning the application of  \cite[Theorem 5, p.~8]{budde2021},
the \cite[Example 3.1, p.~14-15]{budde2021} of rank-one perturbations can be repaired, at least, 
under the additional assumption that $(X,\gamma)$ is C-sequential and $\gamma=\gamma_{s}$. 
The \cite[Example 3.2, p.~15-21]{budde2021} of the 
Gau{\ss}--Weierstra{\ss} semigroup on the space $X=\textrm{M}(\R)$ of bounded Borel measures on $\R$ 
can be saved as well since $\textrm{M}(\R)$ coincides with the space $\textrm{M}_{\operatorname{t}}(\R)$ 
of bounded Radon measures on the Polish space $\R$ and 
$\gamma:=\gamma(\|\cdot\|_{\textrm{M}_{\operatorname{t}}(\R)},\sigma(\textrm{M}_{\operatorname{t}}(\R),\textrm{C}_{\operatorname{b}}(\R)))$ fulfils that 
$(X,\gamma)$ is C-sequential and $\gamma=\gamma_{s}$ by \prettyref{cor:dual_mixed_top}.
Here $\|\cdot\|_{\textrm{M}_{\operatorname{t}}(\R)}$ denotes the total variation norm on 
$\textrm{M}_{\operatorname{t}}(\R)$ (see e.g.~\cite[p.~543]{kunze2009}).
\end{rem}

We will study the relation between $\gamma$-(equi-)continuity and (equi-)tightness in the forthcoming. 
For that purpose we introduce another kind of mixed topology (see \cite[p.~41]{cooper1978}).

\begin{defn}\label{defn:submixed_top}
Let $(X,\|\cdot\|,\tau)$ be a Saks space and $\mathcal{P}_{\tau}$ a directed system of seminorms 
that generates the topology $\tau$ and fulfils \eqref{eq:saks}. 
For a sequence $(p_{n})_{n\in\N}$ in $\mathcal{P}_{\tau}$ and a 
sequence $(a_{n})_{n\in\N}$ in $(0,\infty)$ with $\lim_{n\to\infty}a_{n}=\infty$ we define the seminorm
\[
 \vertiii{x}_{(p_{n}),(a_{n})}:=\sup_{n\in\N}p_{n}(x)a_{n}^{-1},\quad x\in X.
\]
We denote by $\gamma_s:=\gamma_s(\|\cdot\|,\tau)$ the locally convex Hausdorff topology that is generated by 
the system of seminorms $(\vertiii{\cdot}_{(p_n),(a_n)})$ and call it the \emph{submixed topology}.
\end{defn}

\begin{rem}\label{rem:mixed=submixed}
Let $(X,\|\cdot\|,\tau)$ be a Saks space, $\mathcal{P}_{\tau}$ a directed system of seminorms 
that generates the topology $\tau$ and fulfils \eqref{eq:saks}, $\gamma:=\gamma(\|\cdot\|,\tau)$ the mixed 
and $\gamma_{s}:=\gamma_{s}(\|\cdot\|,\tau)$ the submixed topology.
\begin{enumerate}
\item[(a)] We have $\tau\leq\gamma_s\leq \gamma$ since $\gamma_{s}$ is stronger than $\tau$ by definition 
and coarser than $\gamma$ on $X$ by the first part of the proof of \cite[I.4.5 Proposition, p.~41-42]{cooper1978} 
where $\gamma_s$ is called $\widetilde{\widetilde{\gamma}}$. Moreover, $\gamma_{s}$ has the same convergent 
sequences as $\gamma$ by \cite[I.1.10 Proposition, p.~9]{cooper1978} and \cite[Lemma A.1.2, p.~72]{farkas2003} 
where $\gamma_{s}$ is called $\tau_{m}$.
\item[(b)] If 
 \begin{enumerate}
 \item[(i)] for every $x\in X$, $\varepsilon>0$ and $p\in\mathcal{P}_{\tau}$ there are $y,z\in X$ such that $x=y+z$, 
 $p(z)=0$ and $\|y\|\leq p(x)+\varepsilon$, or 
 \item[(ii)] the $\|\cdot\|$-unit ball $B_{\|\cdot\|}=\{x\in X\;|\; \|x\|\leq 1\}$ is $\tau$-compact,
 \end{enumerate}
then $\gamma=\gamma_s$ due to \cite[I.4.5 Proposition, p.~41-42]{cooper1978}. 
\end{enumerate}
\end{rem}

The submixed topology $\gamma_s$ was introduced in \cite[Theorem 3.1.1, p.~62]{wiweger1961}. 
It also appears under the name mixed topology with symbol $\tau_{m}$ or $\tau_{0}^{\mathcal{M}}$
in \cite[Definition A.1.1, p.~72]{farkas2003}, \cite[Definition 2.4, p.~579]{budde2021a} and 
\cite[Definition 2.1, p.~21]{goldys2001} but usually in a context where $\gamma=\gamma_s$ holds. 

\begin{exa}\label{ex:mixed=submixed}
\begin{enumerate}
\item[(a)] Let $\Omega$ be a completely regular Hausdorff space. 
The Saks space $(\mathrm{C}_{\operatorname{b}}(\Omega),\|\cdot\|_{\infty},\tau_{\operatorname{co}})$ 
fulfils condition (i) of \prettyref{rem:mixed=submixed} (b) and 
$\gamma(\|\cdot\|_{\infty},\tau_{\operatorname{co}})=\gamma_s(\|\cdot\|_{\infty},\tau_{\operatorname{co}})$ 
by \cite[Example D), p.~65-66]{wiweger1961}. 
\item[(b)] Let $(X,\|\cdot\|)$ be a Banach space. 
The Saks space $(X',\|\cdot\|_{X'},\sigma^{\ast})$ fulfils condition (ii) 
of \prettyref{rem:mixed=submixed} (b) with $B_{\|\cdot\|_{X'}}$ and 
$\gamma(\|\cdot\|_{X'},\sigma^{\ast})=\gamma_s(\|\cdot\|_{X'},\sigma^{\ast})$ by \cite[Example E), p.~66]{wiweger1961}.
\item[(c)] Let $(X,\|\cdot\|)$ be a Banach space. Due to \cite[Proposition 3.1, p.~275]{schluechtermann1991} 
$B_{\|\cdot\|_{X'}}$ is $\mu^{\ast}$-compact if and only if 
$X$ is a \emph{Schur space}, i.e.~every $\sigma(X,X')$-convergent sequence is $\|\cdot\|$-convergent 
\cite[p.~253]{fabian2011}. Thus the Saks space $(X',\|\cdot\|_{X'},\mu^{\ast})$ satisfies condition (ii) 
of \prettyref{rem:mixed=submixed} (b) and 
$\gamma(\|\cdot\|_{X'},\mu^{\ast})=\gamma_s(\|\cdot\|_{X'},\mu^{\ast})$ if $X$ is a Schur space. 
\item[(d)] Let $(X,\|\cdot\|_{X})$ and $(Y,\|\cdot\|_{Y})$ be Banach spaces. 
$B_{\|\cdot\|_{\mathcal{L}(X;Y)}}$ is $\tau_{\operatorname{wot}}$-compact 
if and only if $Y$ is reflexive by \cite[Theorem 2.19, p.~1689]{choi2008}. 
$B_{\|\cdot\|_{\mathcal{L}(X;Y)}}$ is $\tau_{\operatorname{sot}}$-compact if and only if $Y$ is finite-dimensional 
by \cite[Theorem 3.15, p.~1699]{choi2008}. 
Thus for the Saks spaces $(\mathcal{L}(X;Y),\|\cdot\|_{\mathcal{L}(X;Y)},\tau_{\operatorname{wot}})$  
and $(\mathcal{L}(X;Y),\|\cdot\|_{\mathcal{L}(X;Y)},\tau_{\operatorname{sot}})$ 
condition (ii) of \prettyref{rem:mixed=submixed} (b) holds and 
$\gamma(\|\cdot\|_{\mathcal{L}(X;Y)},\tau_{\operatorname{wot}})
=\gamma_s(\|\cdot\|_{\mathcal{L}(X;Y)},\tau_{\operatorname{wot}})$ if $Y$ is reflexive,
and $\gamma(\|\cdot\|_{\mathcal{L}(X;Y)},\tau_{\operatorname{sot}})
=\gamma_s(\|\cdot\|_{\mathcal{L}(X;Y)},\tau_{\operatorname{sot}})$ if $Y$ is finite-dimensional, respectively.
If $Y$ is the scalar field of $X$, then the last case is already covered by example (b) above. 
\end{enumerate}
\end{exa}

Concerning example (c), we note that the space of absolutely summing sequences $\ell^{1}:=\ell^{1}(\N)$ 
is a Schur space by \cite[Theorem 5.36, p.~252]{fabian2011}. Another example is the \emph{Lipschitz-free} or 
\emph{Arens-Eells space} $\mathcal{F}(M,\omega\circ\d)$ by \cite[Theorem 4.6, p.~186]{kalton2004} 
where $(M,\d)$ is a metric space and $\omega\colon[0,\infty)\to[0,\infty)$ a continuous 
increasing subadditive function with $\omega(0)=0$, $\omega(t)\geq t$ for $0\leq t\leq 1$ 
and $\lim_{t\to 0\rlim}\omega(t)/t=\infty$, e.g.~$\omega(t):=t^{\alpha}$ when $0<\alpha<1$. 
If $M$ contains a distinguished point (the origin), denoted by $0$, then $\mathcal{F}(M,\omega\circ\d)$ 
is the canonical predual of the space $\mathrm{Lip}(M)$ of Lipschitz continuous functions $f\colon M\to\R$ 
with $f(0)=0$ equipped with the norm 
\[
\|f\|_{\operatorname{Lip}}:=\sup_{x,y\in M,x\neq y}\frac{|f(x)-f(y)|}{(\omega\circ\d)(x,y)}
\]
(see \cite[p.~179]{kalton2004}). If $\omega(t):=t^{\alpha}$ for some $0<\alpha<1$, then $\mathrm{Lip}(M)$ 
is the space of $\alpha$-H\"older continuous functions on $(M,\d)$ that vanish at $0$.

For the next proposition we recall our notation that $\mathcal{L}(X)=\mathcal{L}((X,\|\cdot\|);(X,\|\cdot\|))$.

\begin{prop}\label{prop:mixed_cont}
Let $(X,\|\cdot\|,\tau)$ be a Saks space, $\gamma:=\gamma(\|\cdot\|,\tau)$ the mixed 
and $\gamma_{s}:=\gamma_{s}(\|\cdot\|,\tau)$ the submixed topology.
Consider the following assertions for a linear map $S\colon X\to X$:
\begin{enumerate}
\item[(a)] $S$ is $\gamma_s$-continuous.
\item[(b)] $S$ is $\gamma_s$-$\tau$-continuous.
\item[(c)] $S$ is $(\|\cdot\|,\tau)$-tight.
\item[(d)] $S$ is $\tau$-continuous on $\|\cdot\|$-bounded sets, i.e.~the restricted map 
$S_{\mid B}\colon (B,\tau_{\mid B}) \to (X,\tau)$ is continuous on $B$ for every $\|\cdot\|$-bounded set $B\subset X$.
\item[(e)] $S$ is $\tau$-continuous at zero on the $\|\cdot\|$-unit ball $B_{\|\cdot\|}$.
\item[(f)] $S$ is $\gamma$-$\tau$-continuous.
\item[(g)] $S$ is $\gamma$-continuous.
\end{enumerate}
Then $(a)\Rightarrow (b)\Rightarrow (c)\Rightarrow (d)\Leftrightarrow (e)\Leftrightarrow (f)\Leftarrow (g)$. 
Moreover, if $S\in \mathcal{L}(X)$, then $(f)\Rightarrow (g)$. 
If $S\in \mathcal{L}(X)$ and $\gamma=\gamma_s$, then all seven assertions are equivalent.
\end{prop}
\begin{proof}
In the following proof $\mathcal{P}_{\tau}$ denotes the directed system of seminorms that generates the 
$\tau$-topology and fulfils \eqref{eq:saks}, i.e.~$\|x\|=\sup_{p\in\mathcal{P}_{\tau}}p(x)$ for all $x\in X$, 
from the definition of the submixed topology $\gamma_{s}$ 
(see the note below \prettyref{defn:equitight}).
 
$(a)\Rightarrow (b)$, $(f)\Leftarrow (g)$: Obvious as $\tau$ is coarser than $\gamma_s$ and $\gamma$. 

$(b)\Rightarrow (c)$: Let $S\colon (X,\gamma_s)\to (X,\tau)$ be continuous and $p\in\mathcal{P}_{\tau}$. 
Then there are $(p_{n})_{n\in\N}$ in $\mathcal{P}_{\tau}$, a sequence $(a_{n})_{n\in\N}$ in $(0,\infty)$ 
with $\lim_{n\to\infty}a_{n}=\infty$ and $C>0$ such that for any $x\in X$
\[
p(Sx)\leq C \vertiii{x}_{(p_{n}),(a_{n})}=C\sup_{n\in\N}p_{n}(x)a_{n}^{-1}.
\]
Let $\varepsilon >0$. Then there is $N\in\N$ such that $a_{n}^{-1}\leq \varepsilon/C$ for all $n\geq N$. Since 
$\mathcal{P}_{\tau}$ is directed, there are $\widetilde{p}\in\mathcal{P}_{\tau}$ and $\widetilde{C}\geq 0$ such that
\[
     p(Sx)
\leq C\sup_{1\leq n\leq N}p_{n}(x)a_{n}^{-1}+C\sup_{n\geq N}p_{n}(x)a_{n}^{-1}
\leq C\widetilde{C}\sup_{1\leq n\leq N}a_{n}^{-1}\widetilde{p}(x)+\varepsilon\|x\|
\]
for all $x\in X$ and thus (c) holds.

$(c)\Rightarrow (e)$: Let (c) hold and $\widetilde{\varepsilon}>0$. 
For $p\in\mathcal{P}_{\tau}$ and $\varepsilon:=\widetilde{\varepsilon}/2$ there are 
$\widetilde{p}\in\mathcal{P}_{\tau}$ and $C\geq 0$ by (c) such that 
\[
     p(Sx)
\leq C \widetilde{p}(x)+\varepsilon\|x\|
\leq C \widetilde{p}(x)+\frac{\widetilde{\varepsilon}}{2}
\]
for all $x\in B_{\|\cdot\|}$. We set 
$U_{\widetilde{p}}(0;\tfrac{\widetilde{\varepsilon}}{2C})
:=\{x\in X\;|\;\widetilde{p}(x)\leq\tfrac{\widetilde{\varepsilon}}{2C}\}$
and note that 
\[
     p(Sx)
\leq C \widetilde{p}(x)+\frac{\widetilde{\varepsilon}}{2}
\leq \widetilde{\varepsilon}
\] 
for all $x\in U_{\widetilde{p}}(0;\tfrac{\widetilde{\varepsilon}}{2C})\cap B_{\|\cdot\|}$, 
implying the $\tau$-continuity of $S$ at zero on $B_{\|\cdot\|}$.

$(d)\Leftrightarrow (e)\Leftrightarrow (f)$: The first equivalence is \cite[I.1.8 Lemma p.~8]{cooper1978} combined 
with the fact that a subset $B\subset X$ is $\|\cdot\|$-bounded if and only if there is $n\in\N$ such that
$B\subset n B_{\|\cdot\|}$. Due to \cite[I.1.7 Corollary p.~8]{cooper1978} the equivalence $(d)\Leftrightarrow (f)$ 
holds.

$(f)\Rightarrow (g)$ if $S\in \mathcal{L}(X)$: Let (f) hold and $S\in \mathcal{L}(X)$. 
We know that (f) and (e) are equivalent. Let $\varepsilon >0$ and $p\in\mathcal{P}_{\tau}$. Due to (e) 
there are $\widetilde{p}\in\mathcal{P}_{\tau}$ and $\delta>0$ such that 
$
p(Sx)\leq \varepsilon
$
for all $x\in U_{\widetilde{p}}(0;\delta)\cap B_{\|\cdot\|}$. Since $S\in \mathcal{L}(X)$, we have that 
$\|Sx\|\leq \|S\|_{\mathcal{L}(X)}$ for all $x\in B_{\|\cdot\|}$. 
It follows that $Sx\in U_{p}(0;\varepsilon)\cap \|S\|_{\mathcal{L}(X)}B_{\|\cdot\|}$ for all 
$x\in U_{\widetilde{p}}(0;\delta)\cap B_{\|\cdot\|}$. Due to \cite[I.1.7 Corollary p.~8]{cooper1978} and 
\cite[I.1.8 Lemma p.~8]{cooper1978} this means that $S\colon (X,\gamma)\to (X,\gamma)$ is continuous.
\end{proof}

Let us recall that a topological vector space $(X,\tau)$ is called \emph{convex-sequential} or \emph{C-sequential} 
(see \cite[p.~273]{snipes1973}) if every convex sequentially open subset of $(X,\tau)$ is already open. 
For our next proofs we need a classification of C-sequential Hausdorff locally 
convex spaces. Let $(X,\tau)$ be a Hausdorff locally convex space and $\mathcal{U}^{+}$ 
be the collection of all absolutely convex subsets $U\subset X$ which satisfy the condition that every sequence 
$(x_{n})_{n\in\N}$ in $X$ converging to $0$ is eventually in $U$.
Then $\mathcal{U}^{+}$ is a zero neighbourhood basis for a Hausdorff locally convex topology 
$\tau^{+}\geq\tau$ on $X$, which is the finest Hausdorff locally convex topology on $X$ with the same 
convergent sequences as $\tau$ by \cite[Proposition 1.1, p.~342]{webb1968}. 
In the view of \prettyref{rem:mixed=submixed} (a) we have $\gamma^{+}=\gamma_{s}^{+}$. 

\begin{prop}[{\cite[Theorem 7.4, p.~52]{wilansky1981}}]\label{prop:c_sequential}
Let $(X,\tau)$ be a Hausdorff locally convex space. Then the following assertions are equivalent:
\begin{enumerate}
\item[(a)] $X$ is C-sequential.
\item[(b)] $\tau^{+}=\tau$.
\item[(c)] For any Hausdorff locally convex space $Y$ a linear map $f\colon X\to Y$ is continuous if and only if 
it is sequentially continuous. 
\end{enumerate}
\end{prop}

Every bornological topological vector space is C-sequential by \cite[Theorem 8, p.~280]{snipes1973}. 
However, if $(X,\gamma)$ is bornological (or barrelled) for the mixed topology 
$\gamma:=\gamma(\|\cdot\|,\tau)$, then $\gamma$ coincides with the $\|\cdot\|$-topology 
by \cite[I.1.15 Proposition, p.~12]{cooper1978}. Hence in the interesting cases (for us) the space 
$(X,\gamma)$ is neither bornological nor barrelled. We recall the following sufficient conditions 
for being C-sequential given in 
\cite[Proposition 5.7, p.~2681-2682]{kruse_meichnser_seifert2018} and 
\cite[Corollary 7.6, p.~52]{wilansky1981}.

\begin{prop}\label{prop:mixed_topology_sequential}
Let $(X,\|\cdot\|)$ be a Banach space and $\tau$ a Hausdorff locally convex topology on $X$ that is coarser than the 
$\|\cdot\|$-topology and $\gamma:=\gamma(\|\cdot\|,\tau)$ the mixed topology. If 
\begin{enumerate}
\item[(i)] $\tau$ is metrisable on the $\|\cdot\|$-unit ball $B_{\|\cdot\|}:=\{x\in X\;|\;\|x\|\leq 1\}$, or
\item[(ii)] $(X,\gamma)$ is a Mackey--Mazur space,
\end{enumerate}
then $(X,\gamma)$ is a C-sequential space.
\end{prop}

Recall that a Hausdorff locally convex space $(X,\tau)$ with scalar field $\mathbb{K}=\R$ or $\C$ is called \emph{Mazur space} 
(see \cite[p.~40]{wilansky1981}) or \emph{weakly semi-bornological} (see \cite[p.~337]{beatty1996}) if 
\[
(X,\tau)'=\{ y\,\colon\, X\to\mathbb{K}\;|\;y\;\text{is}\;\tau\text{-sequentially continuous}\}.
\]
Due to \prettyref{prop:c_sequential} every C-sequential space is a Mazur space. 
Mackey--Mazur spaces are also called \emph{semi-bornological spaces} (see \cite[p.~337]{beatty1996}). 

\begin{thm}\label{thm:mixed_cont}
Let $(X,\|\cdot\|,\tau)$ be a triple satisfying \prettyref{ass:standard}, 
$\gamma:=\gamma(\|\cdot\|,\tau)$ the mixed and $\gamma_{s}:=\gamma_{s}(\|\cdot\|,\tau)$ the submixed topology
and $(T(t))_{t\geq 0}$ a $\tau$-bi-continuous semigroup. 
\begin{enumerate}
\item[(a)] If $(X,\gamma)$ is C-sequential, then $T(t)$ is $\gamma$-continuous for every $t\geq 0$. 
\item[(b)] If $(X,\gamma)$ is C-sequential and $\gamma=\gamma_s$, 
then $T(t)$ is $(\|\cdot\|,\tau)$-tight for every $t\geq 0$. 
\end{enumerate}
\end{thm}
\begin{proof}
Due to \prettyref{prop:mixed_cont} and since $T(t)\in\mathcal{L}(X)$ for every $t\geq 0$ we only need to 
prove that $T(t)$ is $\gamma$-$\tau$-continuous for every $t\geq 0$.
Due to \prettyref{prop:c_sequential} and $(X,\gamma)$ being C-sequential 
the continuity of $T(t)\colon (X,\gamma)\to(X,\tau)$ for $t\geq 0$ is equivalent 
to $\gamma$-$\tau$-sequential continuity. 
By \prettyref{rem:mixed_top_bi_cont} (a) a sequence is $\gamma$-convergent if and only if it is 
$\tau$-convergent and $\|\cdot\|$-bounded. Thus the bi-equicontinuity of $(T(t))_{t\geq 0}$ implies that 
$T(t)$ is $\gamma$-$\tau$-sequentially continuous for every $t\geq 0$, which proves our claim.
\end{proof}

Note that \prettyref{thm:mixed_cont} asserts a ``pointwise'' statement for every $t\ge0$ whose proof does 
neither require the semigroup property nor the $\tau$-strong-continuity of the trajectories, 
\prettyref{defn:bi_continuous} (ii).
We proceed with an analogue of \prettyref{prop:mixed_cont} for families of linear maps.

\begin{prop}\label{prop:mixed_equicont}
Let $(X,\|\cdot\|,\tau)$ be a Saks space, $\gamma:=\gamma(\|\cdot\|,\tau)$ the mixed 
and $\gamma_{s}:=\gamma_{s}(\|\cdot\|,\tau)$ the submixed topology.
Consider the following assertions for a family $S:=(S(t))_{t\in I}$ of linear maps $X\to X$:
\begin{enumerate}
\item[(a)] $S$ is $\gamma_s$-equicontinuous.
\item[(b)] $S$ is $\gamma_s$-$\tau$-equicontinuous.
\item[(c)] $S$ is $(\|\cdot\|,\tau)$-equitight.
\item[(d)] $S$ is $\tau$-equicontinuous on $\|\cdot\|$-bounded sets, i.e.~the restricted family
$S_{\mid B}:=(S(t)_{\mid B})_{t\in I}$ is $\tau_{\mid B}$-$\tau$-equicontinuous on $B$ for 
every $\|\cdot\|$-bounded set $B\subset X$.
\item[(e)] $S$ is $\tau$-equicontinuous at zero on the $\|\cdot\|$-unit ball $B_{\|\cdot\|}$.
\item[(f)] $S$ is $\gamma$-$\tau$-equicontinuous.
\item[(g)] $S$ is $\gamma$-equicontinuous.
\end{enumerate}
Then $(a)\Rightarrow (b)\Rightarrow (c)\Rightarrow (d)\Leftrightarrow (e)\Leftrightarrow (f)\Leftarrow (g)$. 
Moreover, if $S\subset\mathcal{L}(X)$ and $\sup_{t\in I}\|S(t)\|_{\mathcal{L}(X)}<\infty$, 
then $(f)\Rightarrow (g)$. 
If $S\subset\mathcal{L}(X)$, $\sup_{t\in I}\|S(t)\|_{\mathcal{L}(X)}<\infty$ and $\gamma=\gamma_s$, 
then all seven assertions are equivalent.
\end{prop}
\begin{proof}
This follows analogously to \prettyref{prop:mixed_cont} since \cite[I.1.7 Corollary p.~8]{cooper1978} and 
\cite[I.1.8 Lemma p.~8]{cooper1978} cover the equicontinuous case as well.
\end{proof}

As an application of the preceding proposition we get our main theorem. 

\begin{thm}\label{thm:mixed_equicont_main}
Let $(X,\|\cdot\|,\tau)$ be a triple satisfying \prettyref{ass:standard}, 
$\gamma:=\gamma(\|\cdot\|,\tau)$ the mixed and $\gamma_{s}:=\gamma_{s}(\|\cdot\|,\tau)$ the submixed topology
and $(T(t))_{t\geq 0}$ a $\tau$-bi-continuous semigroup. 
\begin{enumerate}
\item[(a)] If $(X,\gamma)$ is C-sequential, then $(T(t))_{t\geq 0}$ is quasi-$\gamma$-equicontinuous. 
\item[(b)] If $(X,\gamma)$ is C-sequential and $\gamma=\gamma_s$, 
then $(T(t))_{t\geq 0}$ is quasi-$(\|\cdot\|,\tau)$-equitight. 
\end{enumerate}
\end{thm}
\begin{proof}
Part (a) is just a combination of results that are already known. 
If $(X,\gamma)$ is C-sequential, then $(T(t))_{t\geq 0}$ is an \emph{SCLE semigroup} w.r.t.~$\gamma$ 
by \cite[Theorem 7.4, p.~180]{kraaij2016} combined with \prettyref{prop:c_sequential}, which means 
that the semigroup $(T(t))_{t\geq 0}$ is $\gamma$-strongly continuous and locally $\gamma$-equicontinuous 
(see \cite[p.~160]{kraaij2016}). Due to \cite[Corollary 4.7, p.~165]{kraaij2016}, 
\cite[Definition 4.8, p.~165]{kraaij2016}, \cite[Proposition 7.3, p.~179]{kraaij2016} in combination 
with $(X,\gamma)$ being C-sequential and \cite[Corollary 6.5, p.~176]{kraaij2016} 
the semigroup $(T(t))_{t\geq 0}$ is quasi-$\gamma$-equicontinuous. 

Let us turn to part (b). By part (a) there is $\alpha_{1}\in\R$ such that the rescaled semigroup
$(e^{-\alpha_{1} t}T(t))_{t\geq 0}$ is $\gamma$-equicontinuous. 
Choosing $\alpha_{2}>\omega_{0}(T)$, we note that 
$(e^{-\alpha_{2} t}T(t))_{t\geq 0}$ is bounded w.r.t.~$\|\cdot\|_{\mathcal{L}(X)}$ by 
the exponential boundedness of $(T(t))_{t\geq 0}$. Hence we get with $\alpha:=\max(\alpha_{1},\alpha_{2})$ that
$(S(t))_{t\geq 0}:=(e^{-\alpha t}T(t))_{t\geq 0}$ is $\gamma$-equicontinuous and 
$\sup_{t\geq 0}\|S(t)\|_{\mathcal{L}(X)}<\infty$. It follows from \prettyref{prop:mixed_equicont} 
that $(S(t))_{t\geq 0}$ is $(\|\cdot\|,\tau)$-equitight and thus $(T(t))_{t\geq 0}$ is 
quasi-$(\|\cdot\|,\tau)$-equitight.
\end{proof}

Part (b) gives a partial answer to the question raised in \cite[Remark 2.5, p.~92-93]{farkas2004}, namely, 
which other $\tau$-bi-continuous semigroups apart from the $\tau_{\operatorname{co}}$-bi-continuous semigroups 
on $\mathrm{C}_{\operatorname{b}}(\Omega)$, with $\Omega$ Polish, are locally $(\|\cdot\|,\tau)$-equitight. We remark that
the latter semigroups are actually quasi-$(\|\cdot\|_{\infty},\tau_{\operatorname{co}})$-equitight as well 
(see \prettyref{thm:strict_topo_quasi_equicont} below).

The proof of \prettyref{thm:mixed_equicont_main} heavily relies on the quite involved theory of \cite{kraaij2016}. 
Next, we present a simpler and independent proof. On the one hand this proof only works under the stronger assumption 
that $\tau$ is metrisable on the $\|\cdot\|$-unit ball $B_{\|\cdot\|}$, on the other hand this stronger condition 
is actually the one that we check in almost every example to get that $(X,\gamma)$ is C-sequential 
(see \prettyref{rem:mixed_equicont}).
We state this special case of \prettyref{thm:mixed_equicont_main} in the following proposition.
\begin{prop}\label{prop:mixed_equicont_main}
Let $(X,\|\cdot\|,\tau)$ be a triple satisfying \prettyref{ass:standard}, 
$\gamma:=\gamma(\|\cdot\|,\tau)$ the mixed and $\gamma_{s}:=\gamma_{s}(\|\cdot\|,\tau)$ the submixed topology
and $(T(t))_{t\geq 0}$ a $\tau$-bi-continuous semigroup. 
\begin{enumerate}
\item[(a)] If $\tau$ is metrisable on the $\|\cdot\|$-unit ball $B_{\|\cdot\|}$, 
then $(T(t))_{t\geq 0}$ is quasi-$(\|\cdot\|,\tau)$-equicontinuous. 
\item[(b)] If $\tau$ is metrisable on the $\|\cdot\|$-unit ball $B_{\|\cdot\|}$ and $\gamma=\gamma_s$, 
then $(T(t))_{t\geq 0}$ is quasi-$(\|\cdot\|,\tau)$-equitight. 
\end{enumerate}
\end{prop}
\begin{proof}
Since $\tau$ is metrisable on $B_{\|\cdot\|}$, there is a metric 
$\d\colon B_{\|\cdot\|}\times B_{\|\cdot\|}\to[0,\infty)$ which induces the $\tau$-topology on $B_{\|\cdot\|}$. 
We choose $\alpha>\omega_{0}(T)$ and note that the rescaled semigroup 
$(S(t))_{t\geq 0}:=(e^{-\alpha t}T(t))_{t\geq 0}$ 
is globally $\tau$-bi-equicontinuous on $[0,\infty)$ and $\sup_{t\geq 0}\|S(t)\|_{\mathcal{L}(X)}<\infty$ 
by \cite[Proposition 1.4 (b), p.~8]{kuehnemund2001} and the exponential boundedness of $(T(t))_{t\geq 0}$.
Due to \prettyref{prop:mixed_equicont} we only need 
to show that $(S(t))_{t\geq 0}$ is $\gamma$-equicontinuous at zero on $B_{\|\cdot\|}$. 
We prove this by contradiction. Suppose that $(S(t))_{t\geq 0}$ is not $\gamma$-equicontinuous at zero 
on $B_{\|\cdot\|}$. This implies that there are $\varepsilon>0$, a sequence $(x_{n})_{n\in\N}$ 
in $B_{\|\cdot\|}$ and a sequence $(t_{n})_{n\in\N}$ in $[0,\infty)$ such that 
$\d(x_{n},0)<\tfrac{1}{n}$ and $\d(S(t_{n})x_{n},0)>\varepsilon$ for all $n\in\N$. We deduce 
that $(x_{n})_{n\in\N}$ is $\|\cdot\|$-bounded and $\tau$-convergent to zero and that 
$\tau$-$\lim_{n\to\infty}S(t_{n})x_{n}\neq 0$. Thus $(S(t))_{t\geq 0}$ is not globally $\tau$-bi-equicontinuous, 
which is a contradiction.
\end{proof}

Let us turn to the sufficient conditions from \prettyref{prop:mixed_topology_sequential} ensuring that 
$(X,\gamma)$ is C-sequential.

\begin{rem}\label{rem:mixed_equicont}
\begin{enumerate}
\item[(a)] Let $\Omega$ be a completely regular Hausdorff space and 
$X:=\mathrm{C}_{\operatorname{b}}(\Omega)$. 
Then the compact-open topology $\tau_{\operatorname{co}}$ is metrisable on $B_{\|\cdot\|_{\infty}}$ if and only if 
$\Omega$ is \emph{hemicompact} by \cite[Remark II.1.24 1), p.~88]{cooper1978} and \prettyref{defn:mixed_top_Saks} (a).
We recall that a Hausdorff space $\Omega$ is called hemicompact 
if there is a sequence $(K_{n})_{n\in\N}$ of compact sets in $\Omega$ 
such that for every compact set $K\subset\Omega$ there is $N\in\N$ such that $K\subset K_{N}$ 
\cite[Exercises 3.4.E, p.~165]{engelking1989}.
In particular, $(\mathrm{C}_{\operatorname{b}}(\Omega),\gamma)$ is C-sequential if $\Omega$ is hemicompact by 
\prettyref{prop:mixed_topology_sequential} (i) where 
$\gamma:=\gamma(\|\cdot\|_{\infty},\tau_{\operatorname{co}})=\beta_{0}$. 

Furthermore, $(\mathrm{C}_{\operatorname{b}}(\Omega),\gamma)$ is also a C-sequential space 
if $\Omega$ is a Polish space 
by \cite[Theorem 2.4 (a), p.~316]{sentilles1972}, \cite[Theorem 9.1 (a), p.~332]{sentilles1972}, 
\cite[Theorem 8.1 (a), p.~330]{sentilles1972} and \prettyref{prop:c_sequential}. 
\item[(b)] Let $X$ be a Banach space. Then the weak$^{\ast}$-topology $\sigma(X',X)$ is metrisable 
on $B_{\|\cdot\|_{X'}}$ if and only if $X$ is separable by \cite[V.5.2 Theorem, p.~426]{dunford1958}.
In particular, $(X',\gamma)$ is C-sequential if $X$ is separable by 
\prettyref{prop:mixed_topology_sequential} (i) 
where $\gamma:=\gamma(\|\cdot\|_{X'},\sigma(X',X))=\tau_{\operatorname{c}}(X',X)$. 
\item[(c)] Let $(X,\|\cdot\|)$ be a Banach space. 
Then the dual Mackey topology $\mu(X',X)$ is metrisable on $B_{\|\cdot\|_{X'}}$ 
if and only if $X$ is an \emph{SWCG space} by \cite[2.1 Theorem, p.~388-389]{schluechtermann1988}. 
We recall that a Banach space $(X,\|\cdot\|)$ is called strongly weakly compactly generated (SWCG) if there exists 
a $\sigma(X,X')$-compact set $K\subset X$ such that for every $\sigma(X,X')$-compact set $L\subset X$ 
and $\varepsilon>0$ there is $n\in\N$ with $L\subset (nK+\varepsilon B_{\|\cdot\|})$ by
\cite[p.~387]{schluechtermann1988}. 
In particular, $(X',\gamma)$ is C-sequential if $X$ is an SWCG space by 
\prettyref{prop:mixed_topology_sequential} (i) where $\gamma:=\gamma(\|\cdot\|_{X'},\mu(X',X))=\mu(X',X)$. 
\item[(d)] Let $X,Y$ be Banach spaces. 
Then the weak operator topology $\tau_{\operatorname{wot}}$ is metrisable on 
$B_{\|\cdot\|_{\mathcal{L}(X;Y)}}$ if $X$ and $Y'$ are separable 
($\tau_{\operatorname{wot}}$ is metrisable on $B_{\|\cdot\|_{\mathcal{L}(X;Y)}}$ by 
$\d(S,T):=\sum_{m,n=1}^{\infty}2^{-(m+n)}|\langle y_{m}',Tx_{n}-Sx_{n}\rangle|$ for
$S,T\in B_{\|\cdot\|_{\mathcal{L}(X;Y)}}$ where $(x_{n})_{n\in\N}$ 
is a $\|\cdot\|_{X}$-dense sequence in $B_{\|\cdot\|_{X}}$ and $(y_{m}')_{m\in\N}$ a $\|\cdot\|_{Y'}$-dense 
sequence in $B_{\|\cdot\|_{Y'}}$).
In particular, $(\mathcal{L}(X;Y),\gamma)$ is C-sequential if $X$ and $Y'$ are separable by 
\prettyref{prop:mixed_topology_sequential} (i) where 
$\gamma:=\gamma(\|\cdot\|_{\mathcal{L}(X;Y)},\tau_{\operatorname{wot}})$. 

The strong operator topology $\tau_{\operatorname{sot}}$ is metrisable on 
$B_{\|\cdot\|_{\mathcal{L}(X;Y)}}$ if $X$ is separable ($\tau_{\operatorname{sot}}$ is metrisable on 
$B_{\|\cdot\|_{\mathcal{L}(X;Y)}}$ by 
$\d(S,T):=\sum_{n=1}^{\infty}2^{-n}\|Tx_{n}-Sx_{n}\|_{Y}$ for
$S,T\in B_{\|\cdot\|_{\mathcal{L}(X;Y)}}$ where $(x_{n})_{n\in\N}$ 
is a $\|\cdot\|_{X}$-dense sequence in $B_{\|\cdot\|_{X}}$).
In particular, $(\mathcal{L}(X;Y),\gamma)$ is C-sequential if $X$ is separable by 
\prettyref{prop:mixed_topology_sequential} (i) where 
$\gamma:=\gamma(\|\cdot\|_{\mathcal{L}(X;Y)},\tau_{\operatorname{sot}})$. 
\end{enumerate}
\end{rem}

Concerning example (a), every hemicompact space is $\sigma$-compact 
by \cite[Exercises 3.8.C (a), p.~194]{engelking1989} (and this implication is strict), every 
first-countable hemicompact space is locally compact by \cite[Exercises 3.4.E (a), p.~165]{engelking1989}, 
and a locally compact Hausdorff space is hemicompact if and only if it is $\sigma$-compact by 
\cite[Exercises 3.8.C (b), p.~195]{engelking1989}.

Concerning example (c), examples of SWCG spaces are reflexive Banach spaces, separable Schur spaces, 
the space $\mathcal{N}(H)$ of operators of trace class on a separable Hilbert space $H$ and the space 
$L^{1}(\Omega,\nu)$ of (equivalence classes) of absolutely integrable functions on $\Omega$ 
w.r.t.~a $\sigma$-finite measure $\nu$ by \cite[2.3 Examples, p.~389-390]{schluechtermann1988}.

Let us look at the second sufficient condition from \prettyref{prop:mixed_topology_sequential} that 
also guarantees that $(X,\gamma)$ is C-sequential, namely that $(X,\gamma)$ is a Mackey--Mazur space. 
Some of the spaces we already considered fulfil this one as well. 

\begin{rem}\label{rem:mackey_mazur}
\begin{enumerate}
\item[(a)] Condition (ii) of \prettyref{prop:mixed_topology_sequential} 
is fulfilled in \prettyref{rem:mixed_equicont} (a) 
if $\Omega$ is a hemicompact Hausdorff $k_{\R}$\emph{-space} 
or a Polish space since $(\mathrm{C}_{\operatorname{b}}(\Omega),\beta_{0})$ is C-sequential and
\[
 \gamma(\|\cdot\|_{\infty},\tau_{\operatorname{co}})
=\beta_{0}=\mu(\mathrm{C}_{\operatorname{b}}(\Omega),\mathrm{M}_{\operatorname{t}}(\Omega))
\]
by \cite[Theorem 5.2, p.~884]{mosiman1972} and \cite[Theorem 9.1 (a), (d), p.~332]{sentilles1972}
where $\mathrm{M}_{\operatorname{t}}(\Omega)=(\mathrm{C}_{\operatorname{b}}(\Omega),\beta_{0})'$ 
is the space of bounded Radon measures on $\Omega$ by \cite[Theorem 4.4, p.~320]{sentilles1972}. 
Here we use the definition of a Radon measure given in \cite[p.~543]{kunze2009} where the space 
$\mathrm{M}_{\operatorname{t}}(\Omega)$ is called $\mathcal{M}_{0}(\Omega)$, in \cite[p.~312]{sentilles1972} 
this space is just called $M_{t}$ and the index $t$ stands for \emph{tight}.
We recall that a completely regular space $\Omega$ is called $k_{\R}$-space if any map $f\colon\Omega\to\R$ 
(or equivalently, for any completely regular space $Y$ and any map $f\colon\Omega\to Y$)
whose restriction to each compact $K\subset\Omega$ is continuous, the map is already continuous on $\Omega$ 
(see \cite[p.~487]{michael1973} and \cite[(2.3.7) Proposition, p.~22]{buchwalter1969}). 

With regard to strong Mackey spaces we mentioned in the introduction, 
we note that $(\mathrm{C}_{\operatorname{b}}(\Omega),\beta_{0})$ is a \emph{strong Mackey space}, 
i.e.~$\sigma(\mathrm{M}_{\operatorname{t}}(\Omega),\mathrm{C}_{\operatorname{b}}(\Omega))$-compact subsets 
of $\mathrm{M}_{\operatorname{t}}(\Omega)$ are $\beta_{0}$-equicontinuous (see \cite[p.~317]{sentilles1972}), 
and that $\Omega$ is $\beta$\emph{-simple}, i.e.~the strict topologies $\beta_{0}$, $\beta$ and $\beta_{1}$ 
of Sentilles \cite[p.~314-315]{sentilles1972} coincide on $\mathrm{C}_{\operatorname{b}}(\Omega)$ 
(see \cite[Definition 2.12,p.~877]{mosiman1972}), 
if $\Omega$ is a hemicompact Hausdorff $k_{\R}$-space or a Polish space 
by \cite[Theorem 5.2, p.~884]{mosiman1972} and \cite[Theorems 5.7, 5.8 (b), 9.1 (a), p.~325, 332]{sentilles1972}.
\item[(b)] Clearly, condition (ii) of \prettyref{prop:mixed_topology_sequential} is also fulfilled 
in \prettyref{rem:mixed_equicont} (c) if $X$ is an SWCG space.
\item[(c)] Let $(X,\|\cdot\|)$ be a Banach space. If $X$ is \emph{weakly sequentially complete}, 
i.e.~$(X,\sigma(X,X'))$ is sequentially complete (see \cite[Definition 2.3.10, p.~38]{albiac2006}), 
and has an \emph{almost shrinking basis}, which means that 
$X$ has a Schauder basis such that its associated sequence of coefficient functionals forms a Schauder basis 
of $(X',\mu(X',X))$ (see \cite[p.~75]{kalton1973}), then $(X',\mu(X',X))=(X',\gamma)$ 
with $\gamma:=\gamma(\|\cdot\|_{X'},\mu(X',X))=\mu(X',X)$ 
is a Mackey--Mazur space by \cite[Theorem 7.1, p.~51]{wilansky1981} and thus condition (ii) 
of \prettyref{prop:mixed_topology_sequential} is fulfilled.
\end{enumerate}
\end{rem}

Concerning example (a) above, we will take a closer look at $k_{\R}$-spaces in the next section.
Concerning example (c), SWCG spaces and the space $L^{1}(\Omega,\nu)$ for an arbitrary measure $\nu$ 
are weakly sequentially complete by \cite[2.5 Theorem, p.~390]{schluechtermann1988} and 
\cite[Exercise 5.85, p.~286]{fabian2011}. Moreover, if a Banach space has an unconditional basis 
(see \cite[Definition 3.3.1, p.~51]{albiac2006}), then 
this basis is almost shrinking by \cite[Proposition 4, p.~78]{kalton1973}. The space 
$L^{1}([0,1],\lambda)$, $\lambda$ the Lebesgue measure, has no unconditional basis 
by \cite[Theorem 6.5.3, p.~144]{albiac2006}, 
but the separable Schur, thus SWCG, space $\ell^{1}$ has an unconditional basis, namely, the canonical unit sequences. 
Hence $(\ell^{\infty},\mu(\ell^{\infty},\ell^{1}))$ is a Mackey--Mazur space 
(cf.~\cite[Example 7.3, p.~51]{wilansky1981} and \cite[E.1.1 Proposition, p.~338]{beatty1996}). 
An example of a weakly sequentially complete Banach space with an unconditional basis that is not an SWCG space 
is given in \cite[2.6 Example, p.~391]{schluechtermann1988}.

\begin{rem}\label{rem:assumption_budde2021a}
Let $(X,\|\cdot\|,\tau)$ be a triple satisfying \prettyref{ass:standard}, 
$\gamma:=\gamma(\|\cdot\|,\tau)$ the mixed and $\gamma_{s}:=\gamma_{s}(\|\cdot\|,\tau)$ the submixed topology
and $(T(t))_{t\geq 0}$ a $\tau$-bi-continuous semigroup. 
In \cite[Assumption 2.6, p.~579]{budde2021a} the assumption is made that
\begin{enumerate}
\item[(i)] $(X,\gamma_{s})$ is complete, and 
\item[(ii)] $(T(t))_{t\geq 0}$ locally $\gamma_{s}$-equicontinuous ($\gamma_{s}$ is called $\gamma$ 
in \cite{budde2021a}).
\end{enumerate}
With regard to condition (i) we note that 
the space $(X,\gamma)$ is complete if and only if $B_{\|\cdot\|}$ is $\tau$-complete 
by \cite[I.1.14 Proposition, p.~11]{cooper1978}. In particular, $(X,\gamma)$ is complete 
if $B_{\|\cdot\|}$ is $\tau$-compact, which is a sufficient condition for $\gamma=\gamma_{s}$ 
by \prettyref{rem:mixed=submixed} (b), and fulfilled in \prettyref{ex:mixed=submixed} (b)-(d).

In \prettyref{ex:mixed=submixed} (a) we have $\beta_{0}=\gamma=\gamma_{s}$ as well and 
the space $(X,\gamma_{s})=(\mathrm{C}_{\operatorname{b}}(\Omega),\beta_{0})$ for a completely regular Hausdorff space 
$\Omega$ is complete if and only if $\Omega$ is a $k_{\R}$-space by 
\cite[II.1.9 Corollary, p.~81]{cooper1978} (cf.~\cite[3.6.9 Theorem, p.~72]{jarchow1981}). 
 
In \prettyref{rem:mackey_mazur} (c) the space $(X,\gamma):=(X_{0}',\gamma)=(X_{0}',\mu(X_{0}',X_{0}))$ 
for a Banach space $(X_{0},\|\cdot\|_{0})$ is also complete by \cite[p.~74]{kalton1973}. 
If, in addition, $\gamma_{s}=\mu(X_{0}',X_{0})$ holds, then $(X,\gamma_{s})=(X_{0}',\gamma_{s})$ is complete, too.

Condition (ii) is fulfilled, $(T(t))_{t\geq 0}$ is even quasi-$(\|\cdot\|,\tau)$-equitight, 
by \prettyref{thm:mixed_equicont_main} (b) whenever $(X,\gamma)$ is C-sequential and $\gamma=\gamma_{s}$
(thus for the examples in the intersection of \prettyref{ex:mixed=submixed} and \prettyref{rem:mixed_equicont} 
resp.~\prettyref{rem:mackey_mazur}).
\end{rem}

In the last part of this section we tackle the problem from \prettyref{rem:consequences_of_missing_assump} 
to show that the space $\mathrm{M}_{\operatorname{t}}(\R)$ equipped with the mixed topology is C-sequential and 
that the mixed topology and submixed topology coincide on $\mathrm{M}_{\operatorname{t}}(\R)$.

\begin{prop}\label{prop:dual_mixed_top}
Let $(X,\|\cdot\|,\tau)$ be a Saks space, $\gamma:=\gamma(\|\cdot\|,\tau)$ the mixed topology, 
$X_{\gamma}:=(X,\gamma)$ and $\|\cdot\|_{X_{\gamma}'}$ the restriction of $\|\cdot\|_{X'}$ to $X_{\gamma}'$. 
\begin{enumerate}
\item[(a)] Then $(X_{\gamma}',\|\cdot\|_{X_{\gamma}'},\sigma(X_{\gamma}',X))$ is a Saks space, 
$B_{\|\cdot\|_{X_{\gamma}'}}$ is $\sigma(X_{\gamma}',X)$-compact and 
\[
\gamma':=\gamma(\|\cdot\|_{X_{\gamma}'},\sigma(X_{\gamma}',X))
=\gamma_{s}(\|\cdot\|_{X_{\gamma}'},\sigma(X_{\gamma}',X))
=\tau_{\operatorname{c}}(X_{\gamma}',(X,\|\cdot\|))
\]
where $\tau_{\operatorname{c}}(X_{\gamma}',(X,\|\cdot\|))$ is the topology of uniform convergence  
on compact subsets of $(X,\|\cdot\|)$.
\item[(b)] If $(X,\gamma)$ is complete and 
\begin{enumerate}
\item[(i)] $(X,\gamma)$ is a strong Mackey space, or
\item[(ii)] $(X_{\gamma}',\gamma')$ is a C-sequential space and every $\gamma'$-null sequence in $X_{\gamma}'$ 
is $\gamma$-equicontinuous, 
\end{enumerate}
then 
\[
\gamma'=\tau_{\operatorname{c}}(X_{\gamma}',X_{\gamma}).
\]
\end{enumerate}
\end{prop}
\begin{proof}
(a) The space $X_{\gamma}'$ is a closed subspace of $(X',\|\cdot\|_{X'})$ 
by \cite[I.1.18 Proposition, p.~15]{cooper1978} and thus $(X_{\gamma}',\|\cdot\|_{X_{\gamma}'})$ is a Banach space. 
Since $(X_{\gamma}',\sigma(X_{\gamma}',X))'=X$ is norming for $(X_{\gamma}',\|\cdot\|_{X_{\gamma}'})$, we deduce  
from \prettyref{rem:mixed_top} (c) that $(X_{\gamma}',\|\cdot\|_{X_{\gamma}'},\sigma(X_{\gamma}',X))$ is a Saks space. 
Hence the unit-ball $B_{\|\cdot\|_{X_{\gamma}'}}=\{y\in X_{\gamma}'\;|\;\|y\|_{X_{\gamma}'}\leq 1\}$ is 
$\sigma(X_{\gamma}',X)$-closed by \cite[I.3.1 Lemma, p.~27]{cooper1978}. 
Furthermore, the unit-ball $B_{\|\cdot\|_{X'}}$ is 
$\sigma(X',X)$-compact by the Banach--Alaoglu theorem, which implies that 
$B_{\|\cdot\|_{X_{\gamma}'}}$ is $\sigma(X_{\gamma}',X)$-compact because 
\[
B_{\|\cdot\|_{X_{\gamma}'}}=B_{\|\cdot\|_{X'}}\cap X_{\gamma}'
\quad\text{and}\quad
\sigma(X_{\gamma}',X)=\sigma(X',X)_{\mid X_{\gamma}'}.
\]
Therefore condition (ii) of \prettyref{rem:mixed=submixed} (b) is fulfilled for $\gamma'$ and we obtain 
\[
 \gamma'
=\gamma(\|\cdot\|_{X_{\gamma}'},\sigma(X_{\gamma}',X))
=\gamma_{s}(\|\cdot\|_{X_{\gamma}'},\sigma(X_{\gamma}',X))
\]
as well as 
\[
 \gamma'
=\gamma({\|\cdot\|_{X'}}_{\mid X_{\gamma}'},\sigma(X',X)_{\mid X_{\gamma}'})
=\tau_{\operatorname{c}}(X',(X,\|\cdot\|))_{\mid X_{\gamma}'}
=\tau_{\operatorname{c}}(X_{\gamma}',(X,\|\cdot\|))
\]
by \cite[I.4.6 Proposition, p.~44]{cooper1978} and \prettyref{ex:mixed_top} (b).

(b) First, we prove $\gamma'\leq \tau_{\operatorname{c}}(X_{\gamma}',X_{\gamma})$. 
A compact set in $(X,\|\cdot\|)$ is also compact in $X_{\gamma}$ because $\gamma$ is coarser than the 
$\|\cdot\|$-topology. This yields that the identity map $\operatorname{id}\colon (X_{\gamma}',\tau_{\operatorname{c}}(X_{\gamma}',X_{\gamma}))\to (X_{\gamma}',\tau_{\operatorname{c}}(X_{\gamma}',(X,\|\cdot\|)))$ is 
continuous and therefore we have $\gamma'=\tau_{\operatorname{c}}(X_{\gamma}',(X,\|\cdot\|))
\leq \tau_{\operatorname{c}}(X_{\gamma}',X_{\gamma}))$ by part (a).

Next, we show $\gamma'\geq \tau_{\operatorname{c}}(X_{\gamma}',X_{\gamma})$, i.e.~that 
$\id\colon (X_{\gamma}',\gamma')\to(X_{\gamma}',\tau_{\operatorname{c}}(X_{\gamma}',X_{\gamma}))$ 
is continuous. 

(i) Suppose that $(X,\gamma)$ is a strong Mackey space. Due to \cite[I.1.8 Lemma p.~8]{cooper1978} 
and \cite[I.1.7 Corollary p.~8]{cooper1978} we only need to show that the restriction of 
$\id$ to $B_{\|\cdot\|_{X_{\gamma}'}}$ is 
$\sigma(X_{\gamma}',X)$-$\tau_{\operatorname{c}}(X_{\gamma}',X_{\gamma})$-continuous at zero. 
By part (a) we know that $B_{\|\cdot\|_{X_{\gamma}'}}$ is $\sigma(X_{\gamma}',X)$-compact. 
Then the $\sigma(X_{\gamma}',X)$-compact set 
$B_{\|\cdot\|_{X_{\gamma}'}}$ is $\gamma$-equicontinuous because $(X,\gamma)$ is a strong Mackey space. 
Since $X_{\gamma}=(X,\gamma)$ is complete, the topology $\tau_{\operatorname{c}}(X_{\gamma}',X_{\gamma})$ coincides 
with the topology of uniform convergence on precompact subsets of $X_{\gamma}$ by 
\cite[Theorem 3.5.1, p.~64]{jarchow1981}. 
Hence $\sigma(X_{\gamma}',X)$ and $\tau_{\operatorname{c}}(X_{\gamma}',X_{\gamma})$ coincide on the 
$\gamma$-equicontinuous set $B_{\|\cdot\|_{X_{\gamma}'}}$ by \cite[Theorem 8.5.1 (b), p.~156]{jarchow1981}, 
implying that the restriction of $\id$ to 
$B_{\|\cdot\|_{X_{\gamma}'}}$ is $\sigma(X_{\gamma}',X)$-$\tau_{\operatorname{c}}(X_{\gamma}',X_{\gamma})$-continuous.

(ii) Suppose that $(X_{\gamma}',\gamma')$ is C-sequential and every $\gamma'$-null sequence in $X_{\gamma}'$ 
is $\gamma$-equicontinuous. We only need to show that 
$\id\colon (X_{\gamma}',\gamma')\to(X_{\gamma}',\tau_{\operatorname{c}}(X_{\gamma}',X_{\gamma}))$ 
is sequentially continuous at zero. Then linearity yields that $\id$ is sequentially continuous, which 
implies the continuity of $\id$ by virtue of $(X_{\gamma}',\gamma')$ being C-sequential.

Let $(y_{n})_{n\in\N}$ be a $\gamma'$-null sequence in $X_{\gamma}'$. 
It follows that $(y_{n})_{n\in\N}$ is a $\sigma(X_{\gamma}',X)$-null sequence and 
thus a $\tau_{\operatorname{c}}(X_{\gamma}',X_{\gamma})$-null sequence as well 
because $\sigma(X_{\gamma}',X)$ and  $\tau_{\operatorname{c}}(X_{\gamma}',X_{\gamma})$ coincide
on the $\gamma$-equicontinuous set $(y_{n})_{n\in\N}$. 
Hence $\id$ is sequentially $\gamma'$-$\tau_{\operatorname{c}}(X_{\gamma}',X_{\gamma})$-continuous at zero. 
\end{proof}

We note that one cannot use \cite[Example E), p.~66]{wiweger1961} (or \cite[I.2 Examples A, p.~20-21]{cooper1978}) 
directly to show that $\gamma'=\tau_{\operatorname{c}}(X_{\gamma}',(X,\|\cdot\|))$ in (a)
resp.~$\gamma'=\tau_{\operatorname{c}}(X_{\gamma}',X_{\gamma})$ in (b) since the space 
$(X,\gamma)$ need not be a Banach (or Fr\'echet) space. 
For the next corollary we recall that $\|\cdot\|_{\mathrm{M}_{\operatorname{t}}(\Omega)}$ denotes 
the total variation norm on the space $\mathrm{M}_{\operatorname{t}}(\Omega)$ of bounded Radon measures 
on a completely regular Hausdorff space $\Omega$. 

\begin{cor}\label{cor:dual_mixed_top}
Let $\Omega$ be a completely regular Hausdorff space. 
\begin{enumerate}
\item[(a)] Then $(\mathrm{M}_{\operatorname{t}}(\Omega),\|\cdot\|_{\mathrm{M}_{\operatorname{t}}(\Omega)},\sigma(\mathrm{M}_{\operatorname{t}}(\Omega),\mathrm{C}_{\operatorname{b}}(\Omega)))$ is a Saks space, 
$B_{\|\cdot\|_{\mathrm{M}_{\operatorname{t}}(\Omega)}}$ 
is $\sigma(\mathrm{M}_{\operatorname{t}}(\Omega),\mathrm{C}_{\operatorname{b}}(\Omega))$-compact and 
\begin{align*}
\beta_{0}':=&\gamma(\|\cdot\|_{\mathrm{M}_{\operatorname{t}}(\Omega)},\sigma(\mathrm{M}_{\operatorname{t}}(\Omega),\mathrm{C}_{\operatorname{b}}(\Omega)))
=\gamma_{s}(\|\cdot\|_{\mathrm{M}_{\operatorname{t}}(\Omega)},\sigma(\mathrm{M}_{\operatorname{t}}(\Omega),\mathrm{C}_{\operatorname{b}}(\Omega)))\\
=&\tau_{\operatorname{c}}(\mathrm{M}_{\operatorname{t}}(\Omega),(\mathrm{C}_{\operatorname{b}}(\Omega),\|\cdot\|_{\infty})).
\end{align*}
\item[(b)] If $\Omega$ is a hemicompact $k_{\R}$-space or Polish, then 
$(\mathrm{M}_{\operatorname{t}}(\Omega),\beta_{0}')$ is sequentially complete and
\[
\beta_{0}'=\tau_{\operatorname{c}}(\mathrm{M}_{\operatorname{t}}(\Omega),(\mathrm{C}_{\operatorname{b}}(\Omega),\beta_{0})).
\]
\item[(c)] If $\Omega$ is Polish, then 
$\sigma(\mathrm{M}_{\operatorname{t}}(\Omega),\mathrm{C}_{\operatorname{b}}(\Omega))$ is metrisable 
on $B_{\|\cdot\|_{\mathrm{M}_{\operatorname{t}}(\Omega)}}$ and the space 
$(\mathrm{M}_{\operatorname{t}}(\Omega),\beta_{0}')$ is C-sequential.
\end{enumerate}
\end{cor}
\begin{proof}
(a) The triple $(\mathrm{C}_{\operatorname{b}}(\Omega),\|\cdot\|_{\infty},\tau_{\operatorname{co}})$ is a Saks space 
by \prettyref{ex:mixed_top} (a) and it holds 
$\mathrm{M}_{\operatorname{t}}(\Omega)=(\mathrm{C}_{\operatorname{b}}(\Omega),\beta_{0})'$ 
by \cite[Theorem 4.4, p.~320]{sentilles1972} for completely regular Hausdorff spaces $\Omega$ with 
$\beta_{0}=\gamma(\|\cdot\|_{\infty},\tau_{\operatorname{co}})$. 
Further, we have 
\[
\|\nu\|_{\mathrm{M}_{\operatorname{t}}(\Omega)}=\|\nu\|_{(\mathrm{C}_{\operatorname{b}}(\Omega),\|\cdot\|_{\infty})'},
\quad \nu\in \mathrm{M}_{\operatorname{t}}(\Omega),
\]
by \cite[Example 2.4, p.~441]{kunze2011}. Hence \prettyref{prop:dual_mixed_top} (a) implies our statement. 

(b) The space $(\mathrm{C}_{\operatorname{b}}(\Omega),\beta_{0})$ is complete if $\Omega$ 
is Hausdorff $k_{\R}$-space (see \prettyref{rem:assumption_budde2021a}). In particular, Polish spaces 
are Hausdorff $k_{\R}$-spaces by the comments above \prettyref{thm:strict_topo_quasi_equicont}.
In \prettyref{rem:mackey_mazur} (a) we observed that $(\mathrm{C}_{\operatorname{b}}(\Omega),\beta_{0})$ 
is a strong Mackey space if $\Omega$ is a hemicompact Hausdorff $k_{\R}$-space or Polish. Thus we get 
\[
\beta_{0}'=\tau_{\operatorname{c}}(\mathrm{M}_{\operatorname{t}}(\Omega),(\mathrm{C}_{\operatorname{b}}(\Omega),\beta_{0}))
\]
by condition (ii) of \prettyref{prop:dual_mixed_top} (b). The sequential completeness of 
$(\mathrm{M}_{\operatorname{t}}(\Omega),\beta_{0}')$ follows from the sequential completeness of 
$(\mathrm{M}_{\operatorname{t}}(\Omega),\sigma(\mathrm{M}_{\operatorname{t}}(\Omega),\mathrm{C}_{\operatorname{b}}(\Omega)))$ by \cite[Theorems 4.4, 8.7 (a), p.~320, 331]{sentilles1972} and the fact that a hemicompact Hausdorff 
$k_{\R}$-space or Polish space $\Omega$ is $\beta$-simple (see \prettyref{rem:mackey_mazur} (a)).

(c) By part (a) we know that the unit-ball $B_{\|\cdot\|_{\mathrm{M}_{\operatorname{t}}(\Omega)}}$
is $\sigma(\mathrm{M}_{\operatorname{t}}(\Omega),\mathrm{C}_{\operatorname{b}}(\Omega))$-compact. 
Hence we obtain that $\sigma(\mathrm{M}_{\operatorname{t}}(\Omega),\mathrm{C}_{\operatorname{b}}(\Omega))$ 
is metrisable on $B_{\|\cdot\|_{\mathrm{M}_{\operatorname{t}}(\Omega)}}$ 
by \cite[Lemma 2.5, p.~185]{kraaij2016a} because the Polish space $\Omega$ is $\beta$-simple. 
Thus condition (i) of \prettyref{prop:mixed_topology_sequential}
is satisfied and we deduce that $(\mathrm{M}_{\operatorname{t}}(\Omega),\beta_{0}')$ is C-sequential.
\end{proof}

If $\Omega$ is Polish, then $\Omega$ is $\beta$-simple and it follows from part (b) of \prettyref{cor:dual_mixed_top} 
and \cite[Theorem 1.7, p.~183]{kraaij2016a} that $\beta_{0}'$ is the finest topology 
on $\mathrm{M}_{\operatorname{t}}(\Omega)$ that coincides with $\sigma(\mathrm{M}_{\operatorname{t}}(\Omega),\mathrm{C}_{\operatorname{b}}(\Omega))$ on all $\beta_{0}$-equicontinuous subsets of $\mathrm{M}_{\operatorname{t}}(\Omega)$. 
It is an open question whether this is still true if $\Omega$ is a hemicompact Hausdorff $k_{\R}$-space. 
Another open question is whether \prettyref{cor:dual_mixed_top} (c) holds for hemicompact Hausdorff $k_{\R}$-spaces instead of Polish spaces.

\begin{rem}\label{rem:dual_mixed_top}
Let $\Omega$ be a hemicompact Hausdorff $k_{\R}$-space or Polish, and set
\[
\mathrm{C}_{\operatorname{b}}(\Omega)^{\circ}:=\{y\in \mathrm{C}_{\operatorname{b}}(\Omega)'
\;|\;y\;\tau_{\operatorname{co}}\text{-sequentially continuous on } \|\cdot\|_{\infty}\text{-bounded sets}\}.
\]
Then $(\mathrm{C}_{\operatorname{b}}(\Omega),\beta_{0})$ is C-sequential, in particular a Mazur space, 
by \prettyref{rem:mixed_equicont} (a) and thus 
\begin{equation}\label{eq:dual_mixed_top}
 \mathrm{C}_{\operatorname{b}}(\Omega)^{\circ}
=(\mathrm{C}_{\operatorname{b}}(\Omega),\beta_{0})'
=\mathrm{M}_{\operatorname{t}}(\Omega)
\end{equation}
by \cite[I.1.18 Proposition, p.~15]{cooper1978}. 
A sequence is a $\beta_{0}'$-null sequence in $\mathrm{M}_{\operatorname{t}}(\Omega)$ if and only if 
it is a $\sigma(\mathrm{M}_{\operatorname{t}}(\Omega),\mathrm{C}_{\operatorname{b}}(\Omega))$-null sequence and 
$\|\cdot\|_{\mathrm{M}_{\operatorname{t}}(\Omega)}$-bounded by \cite[I.1.10 Proposition, p.~9]{cooper1978}. 
It follows that a $\beta_{0}'$-null sequence in $\mathrm{M}_{\operatorname{t}}(\Omega)$ is
$\beta_{0}$-equicontinuous since $(\mathrm{C}_{\operatorname{b}}(\Omega),\beta_{0})$ is a strong Mackey space
and $B_{\|\cdot\|_{\mathrm{M}_{\operatorname{t}}(\Omega)}}$ 
is $\sigma(\mathrm{M}_{\operatorname{t}}(\Omega),\mathrm{C}_{\operatorname{b}}(\Omega))$-compact 
(see the proof of \prettyref{cor:dual_mixed_top}). 
Therefore every $\beta_{0}'$-null sequence in $\mathrm{M}_{\operatorname{t}}(\Omega)$ 
is $\beta_{0}$-equicontinuous. Hence we conclude that \cite[Hypothesis C, p.~315]{farkas2011}, 
namely, that every $\|\cdot\|_{\mathrm{M}_{\operatorname{t}}(\Omega)} $-bounded 
$\sigma(\mathrm{M}_{\operatorname{t}}(\Omega),\mathrm{C}_{\operatorname{b}}(\Omega))$-null sequence 
in $\mathrm{C}_{\operatorname{b}}(\Omega)^{\circ}$ is $\tau_{\operatorname{co}}$-equicontinuous 
on $\|\cdot\|_{\infty}$-bounded sets,  
is satisfied by \eqref{eq:dual_mixed_top} and \cite[I.1.7 Corollary, p.~8]{cooper1978}.
\end{rem}

\section{Applications}
\label{sect:applications}

For our first application of the results of \prettyref{sect:continuity} 
we consider $\tau_{\operatorname{co}}$-bi-continuous semigroups on $\mathrm{C}_{\operatorname{b}}(\Omega)$ 
for Hausdorff $k_{\R}$-spaces $\Omega$. 
Let us recall some facts on $k_{\R}$-spaces from general topology. 
Examples of $k_{\R}$-spaces are completely regular $k$\emph{-spaces} by \cite[3.3.21 Theorem, p.~152]{engelking1989}.
A topological space $\Omega$ is called $k$-space (compactly generated space) if it satisfies the following condition:
$A\subset \Omega$ is closed if and only if $A\cap K$ is closed in $K$ for every compact $K\subset\Omega$.  
Every locally compact Hausdorff space is a completely regular $k$-space by \cite[p.~152]{engelking1989}. 
Further, every sequential Hausdorff space is a $k$-space by \cite[3.3.20 Theorem, p.~152]{engelking1989}, 
in particular, every first-countable Hausdorff space. Thus metrisable spaces are completely regular $k$-spaces.
However, there are Hausdorff $k_{\R}$-spaces that are not $k$-spaces
by \cite[2.3 Construction, p.~392-393]{noble1969}.\footnote{For instance, 
the space $\Omega:=Y$ in \cite[2.3 Construction, p.~392-393]{noble1969} constructed from 
a completely regular Hausdorff non-$k$-space $X$ due to the comments above this construction. 
For $X$ take e.g.~the \emph{Fortissimo space} which is completely normal by 
\cite[Part II, 25.~Fortissimo space; 1, p.~53-54]{steen1978} and 
easily seen to be Hausdorff, thus completely regular. Further, finite sets are the only compact sets in the 
Fortissimo space $X$ by \cite[Part II, 25.~Fortissimo space; 2, p.~54]{steen1978} 
and hence $X$ is not a $k$-space, which follows from the choice of the non-closed set 
$A:=X\setminus\{p\}$ in the definition of a $k$-space with $p$ from the definition of the Fortissimo space.}
From our observations here and after \prettyref{rem:mixed_equicont} it follows that metrisable spaces 
are $k_{\R}$-spaces and that $\sigma$-compact locally compact 
Hausdorff spaces are hemicompact $k_{\R}$-spaces. Moreover, hemicompact $k_{\R}$-spaces are $k$-spaces by
\cite[Lemme 1 (2.3), p.~55]{buchwalter1972} (cf.~\cite[Lemma 5.1, p.~884]{mosiman1972}). 
Further, there are hemicompact Hausdorff $k_{\R}$-spaces that are neither locally compact nor metrisable by 
\cite[p.~267]{warner1958}, namely the completely regular hemicompact Hausdorff $k$-space 
$\Omega:=(E',\tau_{\operatorname{c}}(E',E))$ for any infinite-dimensional Fr\'echet space $E$. 

Let $\Omega$ be a completely regular Hausdorff space.
The triple $(\mathrm{C}_{\operatorname{b}}(\Omega),\|\cdot\|_{\infty},\tau_{\operatorname{co}})$ is a Saks space 
by \prettyref{ex:mixed_top} (a) and thus \prettyref{ass:standard} (i) and (iii) are fulfilled. 
Further, the space $(\mathrm{C}_{\operatorname{b}}(\Omega),\beta_{0})$ is complete if and only if 
$\Omega$ is a $k_{\R}$-space (see \prettyref{rem:assumption_budde2021a}).\footnote{This means that 
\cite[Proposition 2.7 (iv), p.~4]{farkas2020} is not correct since there are Hausdorff $k_{\R}$-spaces 
which are not $k$-spaces. However, this does have no impact on the validity of the other results 
in \cite{farkas2020}.} 
Therefore it follows from \prettyref{rem:mixed_top} (b) that \prettyref{ass:standard} (ii) is also fulfilled if 
$\Omega$ is a $k_{\R}$-space.

\begin{thm}\label{thm:strict_topo_quasi_equicont}
Let $\Omega$ be a Hausdorff $k_{\R}$-space and 
$(T(t))_{t\geq 0}$ a $\tau_{\operatorname{co}}$-bi-continuous semigroup on $\mathrm{C}_{\operatorname{b}}(\Omega)$. 
If $\Omega$ is hemicompact or Polish, 
then $(T(t))_{t\geq 0}$ is quasi-$\beta_{0}$-equicontinuous and 
quasi-$(\|\cdot\|_{\infty},\tau_{\operatorname{co}})$-equitight.
\end{thm}
\begin{proof}
By virtue of \prettyref{ex:mixed_top} (a) we have $\gamma(\|\cdot\|_{\infty},\tau_{\operatorname{co}})=\beta_{0}$. 
If $\Omega$ is a hemicompact Hausdorff $k_{\R}$-space, 
our statement follows from \prettyref{prop:mixed_equicont_main} combined with 
\prettyref{rem:mixed_equicont} (a) and \prettyref{ex:mixed=submixed} (a). 

If $\Omega$ is a Polish space, our statement follows from \prettyref{thm:mixed_equicont_main} combined with 
\prettyref{rem:mixed_equicont} (a) and \prettyref{ex:mixed=submixed} (a). 
\end{proof}

The preceding theorem improves \cite[Theorem 2.4, Remark 2.5, p.~92-93]{farkas2004} where it is observed that 
every $\tau_{\operatorname{co}}$-bi-continuous semigroup on $\mathrm{C}_{\operatorname{b}}(\Omega)$, $\Omega$ Polish,
is locally $\beta_{0}$-equicontinuous and locally $(\|\cdot\|_{\infty},\tau_{\operatorname{co}})$-equitight. 
In \cite[Example 4.1, p.~320]{farkas2011} (see \prettyref{ex:counterex}) 
an example of a $\tau_{\operatorname{co}}$-bi-continuous semigroup 
$(T(t))_{t\geq 0}$ on $\mathrm{C}_{\operatorname{b}}([0,\omega_{1}))$, where $\omega_{1}$ is the first 
uncountable ordinal and $[0,\omega_{1})$ is equipped with the order topology, is constructed such that 
no $T(t)$, $t>0$, is $\beta_{0}$-continuous. This is no contradiction to \prettyref{thm:strict_topo_quasi_equicont} 
because $[0,\omega_{1})$ is locally compact by \cite[3.3.14 Examples, p.~151]{engelking1989}, 
but neither $\sigma$-compact nor separable by 
\cite[Part II, 42.~Open ordinal space $[0,\Omega)$; 3, 10, p.~69-70]{steen1978}.

\begin{exa}\label{ex:strict_topo_semigroups}
\begin{enumerate}
\item[(a)] The left translation semigroup $(T(t))_{t\geq 0}$ on $\mathrm{C}_{\operatorname{b}}(\R)$ given by 
\[
T(t)f(x):=f(x+t),\quad x\in\R,\,f\in\mathrm{C}_{\operatorname{b}}(\R),\,t\geq 0,
\]
is a $\tau_{\operatorname{co}}$-bi-continuous semigroup by \cite[Examples 6 (b), p.~6-7]{kuehnemund2003}. Due to 
\cite[Example 3.2, p.~549]{kunze2009} it is $\tau_{\operatorname{co}}$-locally equicontinuous but not 
quasi-$\tau_{\operatorname{co}}$-equicontinuous. However, by \prettyref{thm:strict_topo_quasi_equicont} it is 
quasi-$\beta_{0}$-equicontinuous and quasi-$(\|\cdot\|_{\infty},\tau_{\operatorname{co}})$-equitight.
\item[(b)] The Gau{\ss}--Weierstra{\ss} semigroup $(T(t))_{t\geq 0}$ on $\mathrm{C}_{\operatorname{b}}(\R^{d})$ 
given by 
$T(0)f:=f$ and
\[
T(t)f(x):=\frac{1}{(4 \pi t)^{d/2}}\int_{\R^{d}}f(y)e^{\frac{-|y-x|^{2}}{4t}}\d y,
\quad x\in\R^{d},\,f\in\mathrm{C}_{\operatorname{b}}(\R^{d}),\,t> 0,
\]
is $\tau_{\operatorname{co}}$-bi-continuous by \cite[Examples 6 (a), p.~5-6]{kuehnemund2003}. 
It is also quasi-$\beta_{0}$-equicontinuous and quasi-$(\|\cdot\|_{\infty},\tau_{\operatorname{co}})$-equitight by 
\prettyref{thm:strict_topo_quasi_equicont}.
\item[(c)] Next, we consider the Ornstein--Uhlenbeck semigroup on $\mathrm{C}_{\operatorname{b}}(H)$, 
$H$ a separable Hilbert space. 
We refer the reader to \cite[Section 3.3]{kuehnemund2001} and \cite[Section 4]{farkas2004} for details. 
The Ornstein--Uhlenbeck semigroup is $\tau_{\operatorname{co}}$-bi-continuous 
by \cite[Proposition 3.10, p.~63]{kuehnemund2001}. Due to \prettyref{thm:strict_topo_quasi_equicont} 
this semigroup is quasi-$\beta_{0}$-equicontinuous and 
quasi-$(\|\cdot\|_{\infty},\tau_{\operatorname{co}})$-equitight as well.
\item[(d)] The transition semigroups $(T(t))_{t\geq 0}$ on $\mathrm{C}_{\operatorname{b}}(\Omega)$ given by 
\[
T(t)f:=\int_{\Omega}f(\xi)\mu_{t}(\cdot,\d\xi),
\quad f\in\mathrm{C}_{\operatorname{b}}(\Omega),\,t\geq 0,
\]
for a separable Banach space $\Omega$ in \cite[Proposition 6.2, p.~1162]{federico2020} resp.~by
\begin{equation}\label{eq:markov_transition}
T(t)f(\xi):=\mathbb{E}[f(X(t,\xi))],
\quad \xi\in\Omega,\,f\in\mathrm{C}_{\operatorname{b}}(\Omega),\,t\geq 0,
\end{equation}
for a separable Hilbert space $\Omega$ in \cite[Proposition 6.3, p.~1164]{federico2020} 
are $\tau_{\operatorname{co}}$-bi-continuous by \cite[Eq.~(6.3), p.~1162]{federico2020}, 
\cite[Proposition 5.14, p.~1161]{federico2020} and \prettyref{rem:mixed_top_bi_cont} (c) 
($\beta_{0}=\gamma_{s}(\|\cdot\|_{\infty},\tau_{\operatorname{co}})$ is called $\tau_{\mathcal{M}}$ in 
\cite{federico2020}). It follows from \prettyref{thm:strict_topo_quasi_equicont} that 
$(T(t))_{t\geq 0}$ is quasi-$\beta_{0}$-equicontinuous and 
quasi-$(\|\cdot\|_{\infty},\tau_{\operatorname{co}})$-equitight in both cases. 
Interestingly, this is in general not true for the Hausdorff locally convex topology $\tau_{\mathcal{K}}$ 
on $\mathrm{C}_{\operatorname{b}}(\Omega)$ defined in \cite[p.~1154]{federico2020}. 
The topology $\tau_{\mathcal{K}}$ has the same convergent sequences as the (sub)mixed topology $\beta_{0}$ 
by \cite[Proposition 5.4 (i), p.~1154]{federico2020} and \prettyref{rem:mixed_top_bi_cont} (a), 
but $\tau_{\mathcal{K}}$ is strictly coarser than $\beta_{0}$ by \cite[Proposition 4.23, p.~38]{federico2018_arxiv} 
if $\operatorname{dim}(\Omega)\geq 1$. In \cite[Example 6.4, p.~1165]{federico2020} a special case of 
the semigroup \eqref{eq:markov_transition} is constructed which is $\tau_{\mathcal{K}}$-strongly continuous 
and locally sequentially $\tau_{\mathcal{K}}$-equicontinuous but not locally $\tau_{\mathcal{K}}$-equicontinuous 
in contrast to its behaviour w.r.t.~$\beta_{0}$. 
\end{enumerate}
\end{exa}

Keeping the remark before \prettyref{thm:strict_topo_quasi_equicont} in mind that there are Hausdorff
$k_{\R}$-spaces which are not $k$-spaces,
we generalise and improve \cite[Proposition 2.10, p.~5]{farkas2020} on the \emph{Koopman semigroup} 
by weakening the assumptions from $k$-spaces to $k_{\R}$-spaces as well as strengthening its implications 
from local $\beta_{0}$-equicontinuity to local $(\|\cdot\|_{\infty},\tau_{\operatorname{co}})$-equitightness 
resp.~even quasi-$(\|\cdot\|_{\infty},\tau_{\operatorname{co}})$-equitightness.

\begin{cor}\label{cor:koopman} 
Let $\Omega$ be a completely regular Hausdorff space and $\phi\colon [0,\infty)\times\Omega\to\Omega$ 
a \emph{semiflow}, i.e.~with $\phi_{t}(x):=\phi(t,x)$ it holds $\phi_{0}=\id$ and $\phi_{t}\phi_{s}=\phi_{t+s}$ 
for $t,s\geq 0$. Consider the linear Koopman semigroup $(T_{\phi}(t))_{t\geq 0}$ 
on $\mathrm{C}_{\operatorname{b}}(\Omega)$ given by 
\[
T_{\phi}(t)f(x):=f(\phi_{t}(x)),\quad x\in\Omega,\,f\in\mathrm{C}_{\operatorname{b}}(\Omega),\,t\geq 0.
\]
\begin{enumerate}
\item[(a)] If $\phi$ is (jointly) continuous, then $(T_{\phi}(t))_{t\geq 0}$ is $\beta_{0}$-strongly continuous, 
locally $\beta_{0}$-equicontinuous and locally $(\|\cdot\|_{\infty},\tau_{\operatorname{co}})$-equitight. 
\item[(b)] If $\Omega$ is a $k_{\R}$-space and $(T_{\phi}(t))_{t\geq 0}$ is $\beta_{0}$-strongly continuous, 
then $\phi$ is (jointly) continuous.
\item[(c)] If $\Omega$ is a $k_{\R}$-space and $\phi$ (jointly) continuous, then 
$(T_{\phi}(t))_{t\geq 0}$ is a $\tau_{\operatorname{co}}$-bi-continuous semigroup. 
\item[(d)] If $\Omega$ is a hemicompact $k_{\R}$-space or Polish space, and $\phi$ (jointly) continuous, then 
$(T_{\phi}(t))_{t\geq 0}$ is quasi-$\beta_{0}$-equicontinuous and 
quasi-$(\|\cdot\|_{\infty},\tau_{\operatorname{co}})$-equitight.
\end{enumerate}
\end{cor}
\begin{proof}
The proof of part (a) and (b), except for the local $(\|\cdot\|_{\infty},\tau_{\operatorname{co}})$-equitightness, 
is the same as in \cite[Proposition 2.10, p.~5]{farkas2020}. We only have to replace the fact that 
$[0,\infty)\times\Omega$ is a completely regular $k$-space for completely regular $k$-spaces $\Omega$ 
by the fact that $[0,\infty)\times\Omega$ is a $k_{\R}$-space for $k_{\R}$-spaces $\Omega$, which follows from 
a comment after the proof of \cite[Th\'eor\`{e}me (2.1), p.~54-55]{buchwalter1972} since 
$[0,\infty)$ is a locally compact Hausdorff space. 

Let us turn to the local $(\|\cdot\|_{\infty},\tau_{\operatorname{co}})$-equitightness in part (a). We note that 
\[
\|T_{\phi}(t)\|_{\infty}=\sup_{x\in\Omega}|f(\phi_{t}(x))|\leq \|f\|_{\infty},
\quad f\in \mathrm{C}_{\operatorname{b}}(\Omega),
\]
for all $t\geq 0$. Hence we get that 
$(S(t))_{t\in[0,t_{0}]}:=(T_{\phi}(t))_{t\in[0,t_{0}]}$ is 
$(\|\cdot\|_{\infty},\tau_{\operatorname{co}})$-equitight for any $t_{0}\geq 0$ by 
\prettyref{prop:mixed_equicont}, \prettyref{ex:mixed=submixed} (a) and the local $\beta_{0}$-equicontinuity 
of $(T_{\phi}(t))_{t\geq 0}$. Thus $(T_{\phi}(t))_{t\geq 0}$ is locally 
$(\|\cdot\|_{\infty},\tau_{\operatorname{co}})$-equitight.

We turn to part (c). The triple 
$(\mathrm{C}_{\operatorname{b}}(\Omega),\|\cdot\|_{\infty},\tau_{\operatorname{co}})$ 
fulfils \prettyref{ass:standard} if $\Omega$ is a $k_{\R}$-space 
(see the remarks above \prettyref{thm:strict_topo_quasi_equicont}).  
Now, part (a) and (b) and \prettyref{rem:mixed_top_bi_cont} (c) yield 
that $(T_{\phi}(t))_{t\geq 0}$ is $\tau_{\operatorname{co}}$-bi-continuous if $\Omega$ is a $k_{\R}$-space 
and $\phi$ (jointly) continuous.

Finally, part (d) follows from part (c) and \prettyref{thm:strict_topo_quasi_equicont}.
\end{proof}

Part (d) improves \cite[2.3 Theorem, p.~5]{dorroh1993}
where it was already observed that the Koopman semigroup is quasi-$\beta_{0}$-equicontinuous if $\Omega$ is Polish. 
Furthermore, the left translation semigroup in \prettyref{ex:strict_topo_semigroups} (a) is a special case of part (d) 
with $\phi(t,x):=t+x$ and $\Omega:=\R$.

Now, we focus on bi-continuous semigroups on the space $\mathrm{M}_{\operatorname{t}}(\Omega)$ of bounded 
Radon measures on a completely regular Hausdorff space $\Omega$ and generalise 
\cite[Theorems 3.5, 3.6, p.~318]{farkas2011}. 
We recall that for a semigroup $(T(t))_{t\geq 0}$ in $\mathcal{L}(X)$, $X$ Banach, the dual 
semigroup $(T'(t))_{t\geq 0}$ in $\mathcal{L}(X')$ is defined by $\langle T'(t)y,x\rangle:=y(T(t)x)$ 
for $y\in X'$ and $x\in X$.
Due to \prettyref{cor:dual_mixed_top} (a) we get that the triple 
$(\mathrm{M}_{\operatorname{t}}(\Omega),\|\cdot\|_{\mathrm{M}_{\operatorname{t}}(\Omega)},\sigma(\mathrm{M}_{\operatorname{t}}(\Omega),\mathrm{C}_{\operatorname{b}}(\Omega)))$ 
is a Saks space, which yields that \prettyref{ass:standard} (i) and (iii) are fulfilled. 
If $\Omega$ is a hemicompact Hausdorff $k_{\R}$-space or Polish, 
then \prettyref{ass:standard} (ii) is also fulfilled by \prettyref{rem:mixed_top} (b) because 
$(\mathrm{M}_{\operatorname{t}}(\Omega),\beta_{0}')$ is sequentially complete by \prettyref{cor:dual_mixed_top} (b).

\begin{thm}\label{thm:measure_mixed_top_bi_cont}
Let $\Omega$ be a hemicompact Hausdorff $k_{\R}$-space or a Polish space, and 
$(T(t))_{t\geq 0}$ a $\tau_{\operatorname{co}}$-bi-continuous semigroup on $\mathrm{C}_{\operatorname{b}}(\Omega)$. 
Then the semigroup $(T^{\circ}(t))_{t\geq 0}$ defined by 
$T^{\circ}(t):=T'(t)_{\mid \mathrm{M}_{\operatorname{t}}(\Omega)}$ for $t\geq 0$ is a 
$\sigma(\mathrm{M}_{\operatorname{t}}(\Omega),\mathrm{C}_{\operatorname{b}}(\Omega))$-bi-continuous semigroup 
on $\mathrm{M}_{\operatorname{t}}(\Omega)$.
\end{thm}
\begin{proof}
Due to the sequential completeness of $(\mathrm{M}_{\operatorname{t}}(\Omega),\beta_{0}')$ 
and \eqref{eq:dual_mixed_top} we have that \cite[Hypothesis B, p.~314]{farkas2011} is satisfied. 
By \prettyref{rem:dual_mixed_top} we already know that \cite[Hypothesis C, p.~315]{farkas2011} is fulfilled. 
Hence we may apply \cite[Proposition 2.4, p.~315]{farkas2011}, yielding our statement.
\end{proof}

So we get examples of such semigroups $(T^{\circ}(t))_{t\geq 0}$ by using the semigroups 
$(T(t))_{t\geq 0}$ from \prettyref{ex:strict_topo_semigroups} 
and \prettyref{cor:koopman} (d). Actually, all $\sigma(\mathrm{M}_{\operatorname{t}}(\Omega),\mathrm{C}_{\operatorname{b}}(\Omega))$-bi-continuous semigroups 
on $\mathrm{M}_{\operatorname{t}}(\Omega)$ are of this form if $\Omega$ is a 
hemicompact Hausdorff $k_{\R}$-space or Polish since we have the following 
converse of \prettyref{thm:measure_mixed_top_bi_cont}, which is already known in the case that 
$\Omega$ is Polish (see \cite[Theorem 3.6, p.~318]{farkas2011}).

\begin{thm}\label{thm:measure_mixed_top_bi_cont_converse}
Let $\Omega$ be a hemicompact Hausdorff $k_{\R}$-space or a Polish space, and 
$(S(t))_{t\geq 0}$ a $\sigma(\mathrm{M}_{\operatorname{t}}(\Omega),\mathrm{C}_{\operatorname{b}}(\Omega))$-bi-continuous semigroup on $\mathrm{M}_{\operatorname{t}}(\Omega)$.
Then there is a $\tau_{\operatorname{co}}$-bi-continuous semigroup $(T(t))_{t\geq 0}$ on 
$\mathrm{C}_{\operatorname{b}}(\Omega)$ such that $T^{\circ}(t)=S(t)$ for all $t\geq 0$.
\end{thm}
\begin{proof}
The proof is almost the same as for \cite[Theorem 3.6, p.~318]{farkas2011}. We only need to replace the applications 
of \cite[Theorem 3.5, p.~318]{farkas2011} and \cite[Theorem 3.1 c), p.~316-317]{farkas2011} by 
\prettyref{thm:measure_mixed_top_bi_cont} and \prettyref{rem:dual_mixed_top}, respectively, if $\Omega$ is a hemicompact Hausdorff $k_{\R}$-space.
\end{proof}

\begin{thm}\label{thm:measure_mixed_top_equitight}
Let $\Omega$ be a Polish space and $(T(t))_{t\geq 0}$ a 
$\sigma(\mathrm{M}_{\operatorname{t}}(\Omega),\mathrm{C}_{\operatorname{b}}(\Omega))$-bi-continuous semigroup 
on $\mathrm{M}_{\operatorname{t}}(\Omega)$. Then $(T(t))_{t\geq 0}$ is quasi-$\beta_{0}'$-equicontinuous and 
quasi-$(\|\cdot\|_{\mathrm{M}_{\operatorname{t}}(\Omega)},\sigma(\mathrm{M}_{\operatorname{t}}(\Omega),\mathrm{C}_{\operatorname{b}}(\Omega)))$-equitight.
\end{thm}
\begin{proof}
Our statement is a consequence of \prettyref{prop:mixed_equicont_main} and \prettyref{cor:dual_mixed_top}.
\end{proof}

Let us turn to dual semigroups of norm-strongly continuous semigroups. 
If $(X,\|\cdot\|)$ is a Banach space, then 
$(X',\|\cdot\|_{X'},\sigma(X',X))$ is a Saks space by \prettyref{ex:mixed=submixed} (b). 
Thus \prettyref{ass:standard} (i) and (iii) are fulfilled. \prettyref{ass:standard} (ii) is also fulfilled 
by the Banach--Steinhaus theorem (see e.g.~\cite[Corollary, p.~348]{treves2006}). 

\begin{thm}\label{thm:dual_weak_topo_quasi_equicont}
Let $(X,\|\cdot\|)$ be a Banach space and $(T(t))_{t\geq 0}$ a $\|\cdot\|$-strongly continuous semigroup on $X$. 
If $X$ is separable, then the dual semigroup $(T'(t))_{t\geq 0}$ on $X'$ is 
quasi-$\tau_{\operatorname{c}}(X',X)$-equicontinuous and quasi-$(\|\cdot\|_{X'},\sigma(X',X))$-equitight. 
\end{thm}
\begin{proof}
By \cite[Proposition 3.18, p.~78]{kuehnemund2001} $(T'(t))_{t\geq 0}$ is $\sigma(X',X)$-bi-continuous on $X'$. 
Due to \prettyref{ex:mixed_top} (b) we have $\gamma(\|\cdot\|_{X'},\sigma(X',X))=\tau_{\operatorname{c}}(X',X)$. 
For separable $X$ our statement follows from \prettyref{prop:mixed_equicont_main} combined with 
\prettyref{rem:mixed_equicont} (b) and \prettyref{ex:mixed=submixed} (b). 
\end{proof}

The theorem above actually covers all $\sigma(X',X)$-bi-continuous semigroups on $X'$ for separable $X$ 
due to \cite[Theorem 4.2, p.~158]{farkas2004a}, which says that we only need 
$\sigma(X',X)$-continuity on $\|\cdot\|_{X'}$-bounded sets of all members of the $\sigma(X',X)$-bi-continuous semigroup 
for the conclusion that a $\sigma(X',X)$-bi-continuous semigroup is the dual semigroup of a 
$\|\cdot\|$-strongly continuous semigroup on $X$. 
Due to \prettyref{thm:mixed_cont} (a) we have a nice condition on $(X',\tau_{\operatorname{c}}(X',X))$ 
that guarantees this kind of continuity. In particular this condition is satisfied for separable $X$ by 
\prettyref{rem:mixed_equicont} (b). 

\begin{prop}
Let $(X,\|\cdot\|)$ be a Banach space and $(T(t))_{t\geq 0}$ a $\sigma(X',X)$-bi-continuous semigroup. 
If $(X',\tau_{\operatorname{c}}(X',X))$ is C-sequential, then there is a $\|\cdot\|$-strongly continuous semigroup 
$(S(t))_{t\geq 0}$ on $X$ such that $S'(t)=T(t)$ for all $t\geq 0$. 
\end{prop}
\begin{proof}
If $(X',\tau_{\operatorname{c}}(X',X))$ is C-sequential, then by \prettyref{thm:mixed_cont} (a) and 
the equivalence (d)$\Leftrightarrow$(g) of \prettyref{prop:mixed_cont} every $T(t)$ is $\sigma(X',X)$-continuous 
on $\|\cdot\|_{X'}$-bounded sets. Hence the existence of $(S(t))_{t\geq 0}$ follows from 
\cite[Theorem 4.2, p.~158]{farkas2004a}. 
\end{proof}

Let us turn to $\mu(X',X)$-bi-continuous semigroups. 
If $(X,\|\cdot\|)$ is a Banach space, then 
$(X',\|\cdot\|_{X'},\mu(X',X))$ is a Saks space by \prettyref{ex:mixed=submixed} (c). 
Thus \prettyref{ass:standard} (i) and (iii) are fulfilled. Again, \prettyref{ass:standard} (ii) is fulfilled 
by a consequence of the Banach--Steinhaus theorem (see e.g.~\cite[Corollary 2, p.~356]{treves2006}).

\begin{thm}\label{thm:dual_mackey_topo_quasi_equicont}
Let $(X,\|\cdot\|)$ be a Banach space and $(T(t))_{t\geq 0}$ a $\mu(X',X)$-bi-continuous semigroup on $X'$.  
\begin{enumerate}
\item[(a)] If $X$ is an SWCG space, or a weakly sequentially complete space with an almost shrinking basis, 
then $(T(t))_{t\geq 0}$ is quasi-$\mu(X',X)$-equicontinuous. 
\item[(b)] If $X$ is an SWCG Schur space, e.g.~a separable Schur space, then $(T(t))_{t\geq 0}$ 
is quasi-$\mu(X',X)$-equicontinuous and quasi-$(\|\cdot\|_{X'},\mu(X',X))$-equitight.
\end{enumerate}
\end{thm}
\begin{proof}
Due to \prettyref{ex:mixed_top} (c) we have $\gamma(\|\cdot\|_{X'},\mu(X',X))=\mu(X',X)$. 
Part (a) follows from \prettyref{prop:mixed_equicont_main} (a) combined with \prettyref{rem:mixed_equicont} (c) 
if $X$ is an SWCG space. If $X$ is a weakly sequentially complete space with 
an almost shrinking basis, then \prettyref{thm:mixed_equicont_main} (a) in combination 
with \prettyref{rem:mackey_mazur} (c) yields part (a). 

Part (b) is a consequence of \prettyref{prop:mixed_equicont_main} (b) combined with 
\prettyref{rem:mixed_equicont} (c) and \prettyref{ex:mixed=submixed} (c) if $X$ is an SWCG Schur space, 
and the fact that a separable Schur space is an SWCG space by \cite[2.3 Examples (b), p.~389]{schluechtermann1988}.
\end{proof}

Clearly, a Schur space with an almost shrinking basis is already separable. 
Necessary and sufficient conditions such that the dual semigroup of a $\|\cdot\|$-strongly continuous semigroup 
is $\mu(X',X)$-bi-continuous are given in \cite[Propositions 3.19, 3.20, 3.24, p.~78-83]{kuehnemund2001}. 
Due to \cite[Corollary 3.10 (1), p.~553]{kunze2009} such a $\mu(X',X)$-bi-continuous dual semigroup 
is always quasi-$\mu(X',X)$-equicontinuous since it is $\mu(X',X)$-strongly continuous and 
\[
\mathcal{L}(X)=\mathcal{L}(X,\mu(X,X'))=\mathcal{L}(X,\sigma(X,X'))
\] 
where the first equality follows from $(X,\|\cdot\|)$ being a Mackey space, the second from 
\cite[\S 21.4 (6), p.~262]{koethe1969} and 
$\mathcal{L}(X,\tau)$ denotes the space of $\tau$-continuous linear operators from $X$ to $X$. 
However, in return for more restrictions on $X$ \prettyref{thm:dual_mackey_topo_quasi_equicont} 
gives a stronger implication in part (b) and holds for all $\mu(X',X)$-bi-continuous semigroups, 
not only for the ones that are dual semigroups of $\|\cdot\|$-strongly continuous semigroups 
(even though that are the examples we consider). 

\begin{exa}
\begin{enumerate}
\item[(a)] Let $(S(t))_{t\geq 0}$ be the left translation semigroup on $L^{1}(\R):=L^{1}(\R,\lambda)$ given by 
\[
S(t)f(x):=f(x+t),\quad x\in\R,\,f\in L^{1}(\R),\,t\geq 0,
\]
which is $\|\cdot\|_{L^{1}}$-strongly continuous by \cite[Chap.~I, 5.4 Example, p.~39]{engel_nagel2000}.
Its dual semigroup $(T(t))_{t\geq 0}:=(S'(t))_{t\geq 0}$ is the right translation semigroup
on $L^{\infty}(\R)=L^{1}(\R)'$, which fulfils
\[
T(t)f(x)=f(x-t),\quad x\in\R,\,f\in L^{\infty}(\R),\,t\geq 0,
\]
and is a $\mu(L^{\infty},L^{1})$-bi-continuous semigroup by \cite[Examples 3.22, p.~82]{kuehnemund2003}. 
$L^{1}(\R)$ is an SWCG space by \cite[2.3 Examples (d), p.~390]{schluechtermann1988} 
as the Lebesgue measure $\lambda$ is $\sigma$-finite. 
Hence $(T(t))_{t\geq 0}$ is quasi-$\mu(L^{\infty},L^{1})$-equicontinuous 
by \prettyref{thm:dual_mackey_topo_quasi_equicont} (a).
Moreover, $(T(t))_{t\geq 0}$ is quasi-$(\|\cdot\|_{L^{\infty}},\sigma(L^{\infty},L^{1}))$-equitight 
and quasi-$\tau_{\operatorname{c}}(L^{\infty},L^{1})$-equicontinuous by \prettyref{thm:dual_weak_topo_quasi_equicont}. 
\item[(b)] Let $(S(t))_{t\geq 0}$ be the multiplication semigroup on $L^{1}(\R)$ given by 
\[
S(t)f(x):=e^{iq(x)t}f(x),\quad x\in\R,\,f\in L^{1}(\R),\,t\geq 0,
\]
for some measurable and locally integrable function $q\colon\R\to\R$, which is $\|\cdot\|_{L^{1}}$-strongly 
continuous by \cite[Chap.~I, 4.11 Proposition, p.~32]{engel_nagel2000}.
Then the dual semigroup $(T(t))_{t\geq 0}:=(S'(t))_{t\geq 0}$ is a 
$\mu(L^{\infty},L^{1})$-bi-continuous semigroup on $L^{\infty}(\R)$ by \cite[Remark 3.2, p.~83]{kuehnemund2003} 
and quasi-$\mu(L^{\infty},L^{1})$-equicontinuous by \prettyref{thm:dual_mackey_topo_quasi_equicont} (a).
Moreover, $(T(t))_{t\geq 0}$ is quasi-$(\|\cdot\|_{L^{\infty}},\sigma(L^{\infty},L^{1}))$-equitight 
and quasi-$\tau_{\operatorname{c}}(L^{\infty},L^{1})$-equicontinuous by \prettyref{thm:dual_weak_topo_quasi_equicont}. 
\item[(c)] Let $(S(t))_{t\geq 0}$ be the multiplication semigroup on $\ell^{1}=\ell^{1}(\N)$ given by 
\[
S(t)f_{n}:=e^{q(n)t}f_{n},\quad n\in\N,\,f\in\ell^{1},\,t\geq 0,
\]
for some function $q\colon\N\to\C$ with $\sup_{n\in\N}\re q(n)<\infty$, 
which is $\|\cdot\|_{\ell^{1}}$-strongly continuous by 
\cite[Chap.~I, 4.11 Proposition, p.~32]{engel_nagel2000}. Next, we verify that the dual semigroup 
$(T(t))_{t\geq 0}:=(S'(t))_{t\geq 0}$ is a $\mu(\ell^{\infty},\ell^{1})$-bi-continuous semigroup 
on $\ell^{\infty}=\ell^{\infty}(\N)$, using \cite[Propositions 3.19, 3.20, p.~78-81]{kuehnemund2001}. 
We proceed like in \cite[Example 3.22, p.~82]{kuehnemund2001}.
By \cite[Theorem III.2.15, p.~76]{diestel1977} a set $M\subset\ell^{1}$ is $\sigma(\ell^{1},\ell^{\infty})$-compact 
if and only if $M$ is $\|\cdot\|_{\ell^{1}}$-bounded and uniformly absolutely summable, i.e.
\[
\forall\;\varepsilon>0\;\exists\;\delta>0\;\forall\;\Omega\subset\N,\,|\Omega|<\delta,\,f\in M:\;
\sum_{n\in\Omega}|f_{n}|<\varepsilon,
\]
where $|\Omega|$ denotes the cardinality of $\Omega$.
Let $M\subset\ell^{1}$ be $\sigma(\ell^{1},\ell^{\infty})$-compact. Then we have with 
$C:=\sup_{n\in\N}\re q(n)$ that
\[
 \|S(t)f\|_{\ell^{1}}
=\sum_{n=1}^{\infty}e^{t\re q(n)}|f_{n}|
\leq e^{tC}\|f\|_{\ell^{1}}
\leq e^{C}\sup_{g\in M}\|g\|_{\ell^{1}}
<\infty 
\] 
for all $f\in M$ and $0\leq t\leq 1$, which implies that 
\[
M_{1}:=\bigcup_{0\leq t\leq 1}S(t)M
\]
is $\|\cdot\|_{\ell^{1}}$-bounded. Next, we need to show that $M_{1}$ is uniformly absolutely summable. 
Let $\varepsilon>0$. Since $M$ is uniformly absolutely summable, there is $\delta>0$ such that 
for all $\Omega\subset\N$ with $|\Omega|<\delta$, all $f\in M$ and all $0\leq t\leq 1$ it holds 
\[
 \sum_{n\in\Omega}|S(t)f_{n}|
=\sum_{n\in\Omega}e^{t\re q(n)}|f_{n}|
\leq e^{C}\sum_{n\in\Omega}|f_{n}|
<e^{C}\frac{\varepsilon}{e^{C}}
=\varepsilon,
\]
meaning that $M_{1}$ is uniformly absolutely summable. 
Hence $M_{1}$ is $\sigma(\ell^{1},\ell^{\infty})$-compact and it follows 
from \cite[Propositions 3.19, 3.20, p.~78-81]{kuehnemund2001} that 
$(T(t))_{t\geq 0}$ is a $\mu(\ell^{\infty},\ell^{1})$-bi-continuous semigroup 
on $\ell^{\infty}$. The space $\ell^{1}$ is a separable Schur space.
Thus $(T(t))_{t\geq 0}$ is quasi-$\mu(\ell^{\infty},\ell^{1})$-equicontinuous and 
quasi-$(\|\cdot\|_{\ell^{\infty}},\mu(\ell^{\infty},\ell^{1}))$-equitight 
due to \prettyref{thm:dual_mackey_topo_quasi_equicont} (b). Furthermore, $(T(t))_{t\geq 0}$ 
is quasi-$(\|\cdot\|_{\ell^{\infty}},\sigma(\ell^{\infty},\ell^{1}))$-equitight and 
quasi-$\tau_{\operatorname{c}}(\ell^{\infty},\ell^{1})$-equicontinuous 
by \prettyref{thm:dual_weak_topo_quasi_equicont}. 
\end{enumerate}
\end{exa}

Next, we consider implemented semigroups \cite[Sect.~3.2]{alber2001,bratelli1987}. 
If $X$ and $Y$ are Banach spaces, then $(\mathcal{L}(X;Y),\|\cdot\|_{\mathcal{L}(X;Y)},\tau_{\operatorname{wot}})$ and 
$(\mathcal{L}(X;Y),\|\cdot\|_{\mathcal{L}(X;Y)},\tau_{\operatorname{sot}})$ are Saks spaces by 
\prettyref{ex:mixed=submixed} (d). Thus \prettyref{ass:standard} (i) and (iii) are fulfilled in both cases. 
First, by \cite[p.~75]{kuehnemund2001} $\tau_{\operatorname{sot}}$ is sequentially complete on 
$\|\cdot\|_{\mathcal{L}(X;Y)}$-bounded sets. Hence \prettyref{ass:standard} (ii) is fulfilled 
for $\tau_{\operatorname{sot}}$. Second, 
$(\mathcal{L}(X;Y),\gamma(\|\cdot\|_{\mathcal{L}(X;Y)},\tau_{\operatorname{wot}}))$ is complete 
by \cite[I.1.14 Proposition, p.~11]{cooper1978} if $Y$ is reflexive as the unit ball 
$B_{\|\cdot\|_{\mathcal{L}(X;Y)}}$ is $\tau_{\operatorname{wot}}$-compact by \prettyref{ex:mixed=submixed} (d) 
under this condition. Since $(\mathcal{L}(X;Y),\|\cdot\|_{\mathcal{L}(X;Y)},\tau_{\operatorname{wot}})$
is a Saks space, the completeness of 
$(\mathcal{L}(X;Y),\gamma(\|\cdot\|_{\mathcal{L}(X;Y)},\tau_{\operatorname{wot}}))$ yields that
\prettyref{ass:standard} (ii) is fulfilled for $\tau_{\operatorname{wot}}$ 
by \prettyref{rem:mixed_top} (b) if $Y$ is reflexive.

\begin{thm}\label{thm:sop_wop_quasi_equicont}
Let $(X,\|\cdot\|_{X})$, $(Y,\|\cdot\|_{Y})$ be Banach spaces, $(T(t))_{t\geq 0}$ a 
$\|\cdot\|_{Y}$-strongly continuous semigroup on $Y$ and $(S(t))_{t\geq 0}$ a 
$\|\cdot\|_{X}$-strongly continuous semigroup on $X$. We consider the implemented semigroup
$(\mathcal{U}(t))_{t\geq 0}$ on $\mathcal{L}(X;Y)$ given by  
\[
\mathcal{U}(t)R:=T(t)RS(t),\quad t\geq 0,\,R\in\mathcal{L}(X;Y).
\]
\begin{enumerate}
\item[(a)] If $X$ is separable, then $(\mathcal{U}(t))_{t\geq 0}$ is
quasi-$\gamma(\|\cdot\|_{\mathcal{L}(X;Y)},\tau_{\operatorname{sot}})$-equicontinuous. 
\item[(b)] If $X$ is separable and $Y$ finite-dimensional, then $(\mathcal{U}(t))_{t\geq 0}$ is
quasi-$(\|\cdot\|_{\mathcal{L}(X;Y)},\tau_{\operatorname{sot}})$-equitight. 
\item[(c)] If $X$ and $Y$ are separable, $Y$ is reflexive and $(\mathcal{U}(t))_{t\geq 0}$ is locally
$\tau_{\operatorname{wot}}$-bi-equiconti\-nuous, then $(\mathcal{U}(t))_{t\geq 0}$ is  
quasi-$\gamma(\|\cdot\|_{\mathcal{L}(X;Y)},\tau_{\operatorname{wot}})$-equicontinuous and 
quasi-$(\|\cdot\|_{\mathcal{L}(X;Y)},\tau_{\operatorname{wot}})$-equitight. 
\end{enumerate}
\end{thm}
\begin{proof}
By \cite[Proposition 3.16, p.~75]{kuehnemund2001} $(\mathcal{U}(t))_{t\geq 0}$ is a 
$\tau_{\operatorname{sot}}$-bi-continuous semigroup on $\mathcal{L}(X;Y)$. 
Since $\tau_{\operatorname{wot}}\leq \tau_{\operatorname{sot}}$, it is also
$\tau_{\operatorname{wot}}$-strongly continuous and under the assumption of local
$\tau_{\operatorname{wot}}$-bi-equicontinuity in (b) $\tau_{\operatorname{wot}}$-bi-continuous as well. 

Parts (a) and (b) follow from \prettyref{prop:mixed_equicont_main} combined with \prettyref{rem:mixed_equicont} (d) and
and \prettyref{ex:mixed=submixed} (d).
Let us turn to part (c). If $Y$ is reflexive, then the separability of $Y$ implies the separability of $Y'$ 
by \cite[Propositions 6.13, 7.3, p.~49, 53]{meisevogt1997}. Hence we deduce part (c) 
from \prettyref{prop:mixed_equicont_main} (b) combined with \prettyref{rem:mixed_equicont} (d) 
and \prettyref{ex:mixed=submixed} (d). 
\end{proof}

\begin{exa}
Let $H$ be a separable Hilbert space and $\mathcal{K}(H)$ the space of compact operators in 
$\mathcal{L}(H)$. We denote by $\beta$ the Hausdorff locally convex topology 
on $\mathcal{L}(H)$ induced by the directed system of seminorms 
\[
\widetilde{p}_{M}(R):=\max\bigl(\sup_{Q\in M}\|RQ\|_{\mathcal{L}(H)},\sup_{Q\in M}\|QR\|_{\mathcal{L}(H)}\bigr),
\quad R\in \mathcal{L}(H),
\]
for finite $M\subset\mathcal{K}(H)$. Due to \cite[Theorem 3.9, p.~84]{busby1968} and 
\cite[Corollary 2.8, p.~638]{taylor1970} we have $\beta=\mu(\mathcal{L}(H),\mathcal{N}(H))$ where 
$\mathcal{N}(H)$ is the space of trace class operators in $\mathcal{L}(H)=\mathcal{N}(H)'$ 
(cf.~\cite[p.~182]{kraaij2016}). 

Further, we denote by $\tau_{\operatorname{sot}^{\ast}}$ the symmetric
strong operator topology, i.e.~the Hausdorff locally convex topology on $\mathcal{L}(H)$ 
induced by the directed system of seminorms
\[
p_{N}(R):=\max\bigl(\sup_{x\in N}\|Rx\|_{H},\sup_{x\in N}\|R^{\ast}x\|_{H}\bigr),\quad R\in \mathcal{L}(H),
\]
for finite $N\subset H$, and by $\beta_{\operatorname{sot}^{\ast}}$ the mixed topology 
$\gamma(\|\cdot\|_{\mathcal{L}(H)},\tau_{\operatorname{sot}^{\ast}})$ (see \cite[p.~204]{cooper1978} 
where $\beta_{\operatorname{sot}^{\ast}}$ is called $\beta_{\operatorname{s}^{\ast}}$). 
The triple $(\mathcal{L}(H),\|\cdot\|_{\mathcal{L}(H)},\tau_{\operatorname{sot}^{\ast}})$ is a Saks space 
and $(\mathcal{L}(H),\beta_{\operatorname{sot}^{\ast}})$ is complete by \cite[p.~204]{cooper1978}. 
Thus this triple fulfils \prettyref{ass:standard} by \prettyref{rem:mixed_top}. 
Due to \cite[IV.2.11 Proposition, p.~211]{cooper1978} 
we have $\beta_{\operatorname{sot}^{\ast}}=\mu(\mathcal{L}(H),\mathcal{N}(H))$ as well and hence 
\begin{equation}\label{eq:implemented_mackey}
\beta_{\operatorname{sot}^{\ast}}=\mu(\mathcal{L}(H),\mathcal{N}(H))=\beta .
\end{equation}
Since $\mathcal{N}(H)$ is an SWCG space by \cite[2.3 Examples (c), p.~389-390]{schluechtermann1988}, 
we get that $(\mathcal{L}(H),\beta_{\operatorname{sot}^{\ast}})=(\mathcal{L}(H),\beta)$ is a 
C-sequential space by \prettyref{rem:mixed_equicont} (c), which gives a different proof 
of the $\beta^{+}=\beta$ statement in \cite[Proposition 8.5, p.~182]{kraaij2016}.

Let $(T(t))_{t\geq 0}$ be a $\|\cdot\|_{H}$-strongly continuous semigroup on $H$. 
One checks as in the case of $\tau_{\operatorname{sot}}$ (see the proof of 
\cite[Proposition 3.16, p.~75-76]{kuehnemund2001}) that
the implemented semigroup $(\mathcal{U}(t))_{t\geq 0}$ on $\mathcal{L}(H)$ given by  
\begin{equation}\label{eq:implemented_semigroup_hilbert}
\mathcal{U}(t)R:=T^{\ast}(t)RT(t),\quad t\geq 0,\,R\in \mathcal{L}(H),
\end{equation}
is $\tau_{\operatorname{sot}^{\ast}}$-bi-continuous, where $T^{\ast}(t)$ is the adjoint of $T(t)$. 
Since $\beta_{\operatorname{sot}^{\ast}}$ and $\tau_{\operatorname{sot}^{\ast}}$ coincide on 
$\|\cdot\|_{\mathcal{L}(H)}$-bounded sets and 
$(\mathcal{L}(H),\|\cdot\|_{\mathcal{L}(H)},\beta_{\operatorname{sot}^{\ast}})$ 
fulfils \prettyref{ass:standard} by \prettyref{defn:mixed_top_Saks} (a), \prettyref{rem:mixed_top} (b) and 
\cite[Lemma 5.5 (a), p.~2680]{kruse_meichnser_seifert2018}, 
the semigroup $(\mathcal{U}(t))_{t\geq 0}$ is $\beta_{\operatorname{sot}^{\ast}}$-bi-continuous as well. 
Therefore $(\mathcal{U}(t))_{t\geq 0}$ is quasi-$\beta_{\operatorname{sot}^{\ast}}$-equicontinuous 
by \prettyref{thm:dual_mackey_topo_quasi_equicont} (a) because
$\beta_{\operatorname{sot}^{\ast}}=\mu(\mathcal{L}(H),\mathcal{N}(H))$ and $\mathcal{N}(H)$ is 
an SWCG space. 
In addition, we have that $(\mathcal{U}(t))_{t\geq 0}$ is 
quasi-$\gamma(\|\cdot\|_{\mathcal{L}(H)},\tau_{\operatorname{sot}})$-equicontinuous by 
\prettyref{thm:sop_wop_quasi_equicont} (a) as $H$ is separable.
\end{exa}

The preceding example gives an independent proof of \cite[Proposition 8.6, p.~182]{kraaij2016}, 
where it is shown that $(\mathcal{U}(t))_{t\geq 0}$ is an SCLE semigroup w.r.t.~$\beta$, 
due to \eqref{eq:implemented_mackey}. 
The implemented semigroup \eqref{eq:implemented_semigroup_hilbert} is a special case of a so-called 
\emph{quantum dynamical group} by \cite[Example 3.1, p.~30]{fagnola1993}, 
which are defined to be strongly continuous w.r.t.~the ultraweak topology (also called $\sigma$-weak topology).

\bibliography{biblio_tightness}
\bibliographystyle{plainnat}
\end{document}